\documentclass{amsart}
\usepackage{amsmath,amsthm,amsfonts,latexsym,amssymb,amscd,enumerate,times}
\usepackage[plainpages=false]{hyperref}

\numberwithin{equation}{section}

\newcommand{\pd}[2]{\frac {\partial #1}{\partial #2}}

\newcommand{\al}{\alpha}

\newcommand{\la}{\lambda}
\newcommand{\La}{\Lambda}
\newcommand{\oo}{\omega}
\newcommand{\om}{\omega}
\newcommand{\Om}{\Omega}

\newcommand{\dd}{\delta}
\newcommand{\Na}{\nabla}

\newcommand{\ee}{\epsilon}

\newcommand{\te}{\theta}

\newcommand{\beq}{\begin{equation}}
\newcommand{\eeq}{\end{equation}}
\newcommand{\beqs}{\begin{eqnarray*}}
\newcommand{\eeqs}{\end{eqnarray*}}
\newcommand{\beqn}{\begin{eqnarray}}
\newcommand{\eeqn}{\end{eqnarray}}
\newcommand{\beqa}{\begin{array}}
\newcommand{\eeqa}{\end{array}}

\def\as{\underline{S}}
\def\eps{\epsilon}
\def\vphi{\varphi}
\def\tri{\triangle}

\def\td{\tilde}
\def\p{\partial}

\def\RR{{\mathbb R}}

\def\CC{{\mathbb C}}

\def\ri{\rightarrow}

\def\un{\underline}

\def\no{{\nonumber}}

\def\pbp{\sqrt{-1}\partial\bar\partial}
\def\tr{{\rm tr}}

\def\vol{{\rm Vol}}

\def\osc{{\rm osc\,}}

\def\cA{{\mathcal A}}
\def\cC{{\mathcal C}}
\def\cB{{\mathcal B}}

\def\cF{{\mathcal F}}

\def\cH{{\mathcal H}}

\def\cS{{\mathcal S}}

\def\i{{\sqrt{-1}}}

\def\Aut{{\rm Aut}}

\renewcommand\Re{{\rm Re}}
\renewcommand\Im{{\rm Im}}
\def\inv{{\rm inv}}
\def\pr{{\rm pr}}
\newtheorem{prop}{Proposition}[section]
\newtheorem{theo}[prop]{Theorem}
\newtheorem{thm}[prop]{Theorem}
\newtheorem{lem}[prop]{Lemma}

\newtheorem{cor}[prop]{Corollary}
\newtheorem{rem}[prop]{Remark}

\newtheorem{defi}[prop]{Definition}
\newtheorem{defn}[prop]{Definition}

\newcommand{\thmref}[1]{Theorem~\ref{#1}}

\newcommand{\lemref}[1]{Lemma~\ref{#1}}

\title{K\"ahler non-collapsing, eigenvalues and the Calabi flow}

\author{Haozhao Li$^1$  }
\address{ Department of Mathematics, University of Science and Technology
of China, Hefei, 230026, Anhui province, China and Wu Wen-Tsun Key
Laboratory of Mathematics, USTC, Chinese Academy of Sciences, Hefei
230026, Anhui,  China} \email{hzli@ustc.edu.cn}

\author{Kai Zheng}
\address{Institut f\"ur Differentialgeometrie,
Leibniz Universit\"at Hannover, Welfengarten 1, 30167 Hannover,
Germany}\email{zheng@math.uni-hannover.de}

\thanks{$^1$Research
partially supported by NSFC grant No. 11001080 and No. 11131007.}

\subjclass[2000]{Primary 53C44; Secondary 32Q15,  32Q26,  58E11}

\begin{document}
\bibliographystyle{plain}

\date{}

\maketitle

\begin{abstract}
We first prove a new compactness theorem of K\"ahler
metrics, which  confirms  a prediction
in \cite{MR2471594}. Then we establish several eigenvalue estimates along the Calabi flow. Combining the compactness theorem and these eigenvalue estimates, we generalize the method developed for the K\"ahler-Ricci flow in \cite{MR2481736} to obtain
several new small energy theorems of the Calabi flow.
\end{abstract}


\section{Introduction}
In 1982, E. Calabi \cite{MR645743} introduced  extremal K\"ahler (extK)
metrics in a fixed cohomology class
of K\"ahler metrics. They are  critical points of the Calabi energy which is
the $L^2$-norm of the scalar curvature. The K\"ahler-Einstein metrics
and which are more general, the constant scalar curvature K\"ahler (cscK) metrics
 are both extremal K\"ahler
metrics. In the same paper, Calabi also introduced a decreasing flow of the Calabi energy,
which is now well-known as the Calabi flow. The Calabi flow is expected to be
an
effective tool to find cscK metrics in a K\"ahler class.
In this paper, we shall prove a compactness theorem  of K\"ahler metrics
 with its applications on the existence of extK/cscK metrics, and discuss the long time behavior of Calabi flow.

\subsection{K\"ahler non-collapsing}

In order to study the Calabi flow in the frame of geometric
analysis, an important step is to establish a compactness theorem
of K\"ahler metrics under suitable geometric conditions.
Comparing with the compactness properties of Riemannian metrics,
the set of K\"ahler metrics within a fixed cohomological class has
more rigidity. The elliptic
equations of  Riemannian metrics could be written down as
function equations of  K\"ahler potentials. This point of view is
different from the compactness of  Riemannian metrics.

The Cheeger-Gromov theorem \cite{MR0263092}\cite{MR682063} states
that the set of Riemannian metrics with uniform curvature, diameter
upper bound and the volume lower bound has
$C^{1,\alpha}$-compactness up
 to diffeomorphisms.
There are two points of view to treat the compactness of the set of
  K\"ahler metrics. One is to apply the compactness theorem of the
corresponding Riemannian metrics. However, the complex structure
under the diffeomorphisms might "jump" in the limit (see Tian
\cite{MR1055713}). The other   is to consider the point-wise
convergence of the K\"ahler forms and their K\"ahler potentials (see
Ruan \cite{MR1743466}). There are examples that a sequence of Riemannian
metrics converges up to diffeomorphisms, but the corresponding
K\"ahler metrics collapse in some Zariski open set in the point-wise
sense. In particular, within the same K\"ahler class this
phenomenon was called \emph{K\"ahler collapsing} in Chen \cite{MR2471594}.
The $C^{1,\alpha}$-compactness of the set of the K\"ahler metrics
was proved in Chen-He \cite{MR2405167} under uniform Ricci bound
from both sides and the uniform $L^\infty$-bound of the K\"ahler
potential,  and recently in Sz\'ekelyhidi \cite{Szekelyhidi:2012fk} under the
conditions that the Riemannian
curvature is uniformly bounded and the $K$-energy is proper
(which is called $I$-proper in our paper in order to distinguish different notions of properness. See Definition \ref{E and I} in  Section \ref{Energy functional}).

In \cite{MR2471594} Chen initiated the problem that what kinds of
geometric conditions are required to avoid the K\"ahler collapsing?
He showed that the properness of the $K$-energy in a K\"ahler class,
the K\"ahler metrics in a bounded geodesic ball in the space of
K\"ahler potentials, the uniform bound of Ricci curvature and the
uniform diameter are sufficient. Furthermore, he predicted that the
upper bound of the Ricci curvature could be dropped. In general,
when the Ricci curvature only has uniform lower bound, the sequence of
Riemannian metrics could collapse to a lower dimensional manifold
(see Cheeger-Colding
\cite{MR1484888}\cite{MR1815410}\cite{MR1815411}\cite{MR1937830}).
Together with the uniform injectivity radius lower bound and the
uniform volume upper bound, Anderson-Cheeger \cite{MR1158336} proved
the $C^{\alpha}$-compactness of the subset of  Riemmannian metrics.
We have the following K\"ahler noncollapsing theroem, which relaxes the conditions
of Chen's theorem on the two-side bounds of the Ricci curvature.

\begin{theo}\label{compactness}Let $(M, \Om)$ be
an $n$-dimensional compact K\"ahler manifold for which the $K$-energy is
$I$-proper in the K\"ahler class $\Om$. If $\cS$ is
the  set of K\"ahler metrics in $\Om$ satisfying the following
properties:
\begin{itemize}
\item the $K$-energy is bounded;
\item the $L^p$-norm of the scalar curvature is bounded for some $p>n$;
\item the Ricci curvature is bounded from below;
\item the Sobolev constant is bounded;
\end{itemize}
then $\cS$ is compact in $C^{1,\alpha}$-topology of the space of the
K\"ahler metrics for some $\alpha\in(0,1)$. In particular, K\"ahler collapsing does not
occur.
\end{theo}

In order to prove this theorem, we essentially use the equation of the scalar curvature $S$ of the K\"ahler metric $g_\vphi$, which reads
\begin{equation}\label{fflow:csck}
  \left\{
   \begin{aligned}
\frac{\det(g_\vphi)}{\det( g)}&=e^P,\\
\tri_{g_\vphi}
P&=g_{\vphi}^{i\bar{j}}R_{i\bar{j}}(\om)-S(g_\vphi).
   \end{aligned}
  \right.
\end{equation}
This is a second order elliptic system. The first equation is a complex  Monge-Amp\`ere equation with a pseudo-differential term $P$.
The main difficulty of proving this compactness theorem is how to derive the a priori estimates of the second equation which is a   complex linearized  Monge-Amp\`ere equation. Aiming to solve this problem, we first introduce a decomposition method, then apply the De Giorgi-Nash-Moser iteration with the Sobolev constant. The second condition in the theorem suggests us that there might exist analytic singularities for the critical exponent $p=n$.  A nature way to apply the compactness theorem is to find a nice path decreasing the energy, such as the minimizing sequence from the variational direct method, the continuity path from the continuity method (Aubin-Yau path) and the geometric flows including the Calabi flow, the pseudo-Calabi flow, the K\"ahler-Ricci flow,  etc. In this paper, we focus on the discussion of the Calabi flow. Actually, there are large classes of Fano surfaces, where the Sobolev constant is bounded along the Calabi flow (see the end of Section \ref{Calabi flow of the Kahler potentials}).

The $I$-properness of the $K$-energy (cf. Definition \ref{E and I})
was introduced by Tian in \cite{MR1471884} and Tian proved that
the $I$-properness of the $K$-energy in the first Chern class is equivalent
to the existence of
K\"ahler-Einstein metrics when the underlying K\"ahler manifold  has no non-zero
 holomorphic vector fields. The properness condition is a ''coercive" condition
 in the frame of variational theory. However, regarding to the existence of
 cscK/extK metrics, the accurate positive function appearing
  in the ''coercive" condition is still not known clearly. Tian
  \cite{MR1787650} conjectured that in a general K\"ahler class the $I$-properness of the
  $K$-energy is equivalent to the existence of cscK metrics.

Theorem \ref{compactness} has direct applications to the Calabi flow. Since the
$K$-energy is decreasing,
 the first condition appearing in \thmref{compactness} is satisfied automatically
 along the Calabi flow. On the other hand,
according to the functional inequality $\nu(\vphi)-\nu(0) \leq \sqrt{Ca(\vphi)} \cdot d(0,\vphi)$ of Chen \cite{MR2471594}, the uniform bound of the $K$-energy could be replaced by the uniform bounds of the Calabi energy and the geodesic distance.

\subsection{The Calabi flow on the level of K\"ahler potentials}

The Calabi flow is a fourth order flow
of K\"ahler potentials
\begin{align}\label{cf}
\frac{\partial}{\partial t}\varphi &=S(\varphi)-\underline S,
\end{align} where
$\un S$ is the average of the scalar curvature of $\oo.$
On Riemann surfaces, the long time existence and the
convergence of the Calabi flow  have been proved by Chrusical
\cite{MR1101689}  using the Bondi mass estimate. Chen
 \cite{MR1820328} gave a new geometric analysis proof which is based
 on the  compactness properties of the conformal metrics with bounded
 Calabi energy and area.  In \cite{MR1991140} Struwe  gave a
  proof by the integral method without using the maximum
  principle.

On K\"ahler manifolds of higher dimensions, since the Calabi flow is
a fourth order flow, so far for now, restricted to the difficulties
from PDE, not much progress has been made. However, due to 
the intrinsic geometric property of the Calabi flow, Chen
conjectured that the Calabi flow always has long time existence with
any smooth initial K\"ahler metrics. Furthermore the general
conjectural picture of the convergence of the Calabi flow was
outlined in Donaldson \cite{MR2103718}. Assuming the existence of
the critical metrics, the stability problem
studies the asymptotic behavior of the Calabi flow near the critical metrics,
in particular, the exponential convergence of the Calabi flow.
 The first stability theorem of Calabi flow was proved by Chen-He
\cite{MR2405167} near a cscK metric. By using a
different method, Huang and the second
author \cite{MR2953047} proved a stability theorem  when the initial K\"ahler
metric is near an extK metrics. Other stability theorems were proved when
perturbing the complex structure in Chen-Sun \cite{Chen:2010kx} and
in Tosatti-Weinkove \cite{MR2357473} when the first Chern class is
less than or equals to zero. There are also several different kinds
of notions of weak solution of the Calabi flow are constructed by
Berman \cite{Berman:2010uq} in the sense of current in the first
Chern class, approximated by the balancing flow in Fine
\cite{MR2669363} and by the De Giorgi's notion of minimizing movement
in Streets \cite{Streets:2012uq}\cite{Streets:2013fk}.

In this paper, the main problem we concerned is the convergence of the Calabi flow without assuming the existence
of cscK/extK metrics.
The first main technique result of the Calabi flow in this paper is the
following theorem, which assumes that the
initial K\"ahler potential is bounded and   the Calabi energy is sufficiently small.

\begin{thm}\label{theo:main1}
Let $(M, \oo)$ be an compact K\"ahler manifold with vanishing  Futaki
invariant. For any $\la, \La>0$, there is a
 constant $\ee=\ee(\la, \La, \om)$ such
that for any metric $\oo_{\varphi}\in [\oo]$ satisfying \beq
\oo_{\varphi}\geq \la\,\oo,\quad |\varphi|_{C^{3,\al}}\leq \La,
\quad Ca(\oo_{\varphi})\leq \ee, \label{eq0018}\eeq    the   Calabi
flow with the initial metric $\oo_{\varphi}$ exists for all time and
converges exponentially fast to a cscK   metric.

\end{thm}

An analogous result to Theorem \ref{theo:main1} also holds for extK
metrics (cf. Theorem \ref{theo:main1ext}).
We accumulate the discussion above and obtain the corollary of \thmref{compactness} as
the following:
\begin{cor}\label{cor:main}
Let $(M, \om)$ be an $n$-dimensional compact K\"ahler manifold for
which the   $K$-energy is $I$-proper in the K\"ahler class
$\Om.$  If the $L^p$-norm of the  scalar curvature for
some $p>n$, the Sobolev constant and the lower bound of Ricci
curvature are uniformly bounded along the flow, then the
Calabi flow converges exponentially fast to a cscK
metric.

\end{cor}

Combining Theorem \ref{compactness} with Theorem \ref{theo:main1},
we also have the following result using the $K$ energy:

\begin{theo}\label{theo:main1a}
Let $(M, \oo)$ be a $n$-dimensional compact K\"ahler manifold for which the
$K$-energy is $I$-proper.
 For any constants $\la,  \La, K>0$ and $p>n$, there is a constant $\ee=\ee(\la, \La,
 K, p,
 \om)$ such
that if there exists a metric $\oo_{0}\in [\oo]$ satisfying the
following conditions
\begin{align*}  Ric(\oo_0)\geq -\la \,\oo_0,\quad
\|S(\om_0)\|_{L^p(\oo_0)}\leq \La, \quad C_S(\om_0)\leq K,  \quad
\nu_{\oo}(\oo_0)\leq \inf_{\oo'\in \Om}\nu_{\oo}(\oo')+\ee,
\end{align*}
then the Calabi flow with the initial metric $\oo_0$ exists for all
time and converges exponentially fast to a cscK metric.
 \end{theo}

We sketch the proof of Theorem \ref{theo:main1a}. By Theorem \ref{compactness} we obtain a
$C^{3,\alpha}$ K\"ahler potential with small $K$-energy.
Then a contradiction argument implies that the resulting smooth
metric from the Calabi flow has small Calabi energy. This together
with  \thmref{theo:main1}
implies Theorem \ref{theo:main1a}. Analogous results also hold for
  modified Calabi flow and extK metrics (cf.
Theorem \ref{theo:main1aext}).

\subsection{The Calabi flow on the level of K\"ahler metrics}
According to the rich geometries of the Calabi flow, it could be
formulated from different point of views. The  equation \eqref{cf}
of K\"ahler potentials can be written as  the evolution equation of
 K\"ahler metrics
\begin{align}\label{cf metric}
\frac{\partial}{\partial t}\om_\varphi &=\sqrt{-1}\p\bar\p S(\varphi).
\end{align}
In the following, we  study the Calabi flow of the K\"ahler metrics
\eqref{cf metric}. Unlike the K\"ahler Ricci flow or the mean
curvature flow, we have no maximum principle and it is difficult to
obtain the explicit lifespan of the curvature tensor. For this
reason, we introduce the following definition.
\begin{defi}\label{defi:calabiflow}Given any constants  $ \tau, \La>0$.
 A solution $\oo_t\, (t\in [0, T])(T\geq\tau)$ of Calabi flow is called
\begin{itemize}
  \item type $I(\tau, \La)$, if
$$|Rm|(t)\leq \La, \quad \forall\, t\in [0, \tau],$$
\item type $II(\tau, \La)$, if
$$|Rm|(t)+|\p\bar \p S|\leq \La, \quad \forall\, t\in [0, \tau].$$
\end{itemize}

\end{defi}

Starting from a type $(\tau, \La)$ Calabi flow with small Calabi
energy, we show the flow has long time solution and converges to a
cscK metric.

\begin{theo}\label{theo:main2}Let $(M, \oo)$ be an $n$-dimensional compact K\"ahler
manifold with no nonzero holomorphic vector fields. For any $\tau,
\La, K,
\dd>0$, there is a constant $\ee=\ee(\tau, \La, K, \dd, n, \oo)>0$ such that if
the solution $\oo_t$ of the Calabi flow with any initial metric
 $\oo_0\in [\oo]$ satisfies the following properties:
 \begin{enumerate}
   \item[(a)] the Calabi flow $\oo_t$ is of type $II(\tau, \La)$;
   \item[(b)] \beq
C_S(\oo_0)\leq K,\quad \mu_1(\oo_0)\geq \dd, \quad Ca(\oo_0)\leq \ee,
 \eeq
 \end{enumerate}
 then the Calabi flow $\oo_t$  exists for all the time and converges exponentially fast
to a cscK metric.

\end{theo}

The proof of Theorem \ref{theo:main2} essentially depends on the
estimates of the eigenvalue of the linearized operator of the scalar
curvature. We show that under some curvature assumptions the
eigenvalue will decay slowly along the Calabi flow, which guarantees
the exponential decay of the Calabi energy. On the other hand, the
exponential decay of the Calabi energy will force the curvature
tensor bounded along the flow, and the curvature tensor bound will
help to improve the eigenvalue estimates. Repeating this argument,
we can show the Calabi flow exists for all time and converges. The
idea of this proof is developed from \cite{MR2481736}, and here we need to
overcome several difficulties since the linearized operator of the
scalar curvature is a fourth order operator and there is no maximum
principle along the flow.

When $M$ admits non-zero holomorphic vector fields, under the holomorphic action the limit of the orbit of the complex structure might contain larger holomorphic group.
The geometric stability condition (see Definition~\ref{pre-stable}) which is called pre-stable avoids this phenomenon. It is also used in  Chen-Li-Wang \cite{MR2481736} and Phong-Sturm \cite{MR2215459}.

\begin{theo}\label{theo:main3}Let $(M, \oo)$ be a $n$-dimensional compact K\"ahler
manifold with vanishing Futaki invariant.  Assume that $M$ is pre-stable.
 For any $\tau,
\La, K>0$, there is a constant $\ee=\ee(\tau, \La, K, n, \oo)>0$ such that if
the solution $\oo_t$ of the Calabi flow with any initial metric
 $\oo_0\in [\oo]$ satisfies the following properties:
 \begin{enumerate}
   \item[(a)] the Calabi flow $\oo_t$ is of type $I(\tau, \La)$;
   \item[(b)] \beq
C_S(\oo_0)\leq K, \quad Ca(\oo_0)\leq \ee,
 \eeq
 \end{enumerate}
 then the Calabi flow $\oo_t$  exists for all the time and converges exponentially fast
to a cscK metric.\\

\end{theo}

\subsection{The organization}

The organization of this paper is as follows: In Section 2, we will
discuss the energy functionals and prove the K\"ahler non-collapsing theorem.
 In Section 3, we recall
some basic facts on the Calabi flow and the modified Calabi flow.
 In Section 4, we will estimate
the eigenvalue in various cases and prove the exponential decay of
the Calabi energy in a short time interval, which is used to 
obtain uniform bounds on the curvature. In Section 5, we will use the
results in previous sections to show Theorem \ref{theo:main1}, Theorem
\ref{theo:main1a} and  Corollary \ref{cor:main}. In Section 6, we finish the proof of
Theorem \ref{theo:main2} and Theorem \ref{theo:main3}. Moreover, we
will give some conditions on the initial metric which implies the
type
$(\tau, \La)$ Calabi flow solution.\\

 {\bf Acknowledgements}: The authors would like to express their deepest gratitude to Professor Weiyue Ding and Professor Xiuxiong Chen for their constant
encouragement during the past several years. We  would also like to thank Julius Ross for helpful discussion on his paper \cite{MR2221139} and
Weiyong He, Bing Wang and Jinxing Xu for valuable suggestion on this paper.


\section{A K\"ahler non-collapsing theorem}\label{A Kahler non-collapsing theorem}
In this section, we prove that there is no K\"ahler collapsing by
replacing the uniform bound of the Ricci curvature by its lower
bound and the $L^p$ bound of the scalar curvature.

\subsection{Energy functionals}\label{Energy functional}
Let $(M, \oo)$ be a compact K\"ahler manifold with a K\"ahler metric
$\oo.$ The space of K\"ahler potentials is denoted by
$$\cH(M, \oo)=\{\varphi\in C^{\infty}(M, \RR)\;|\;\oo+\pbp \varphi>0\}.$$
Recall that Aubin functions $I$ and $J$ are defined on $\cH(M, \oo)$  by
 \beqs
I_\om(\vphi)&=&\frac{1}{V}\int_{M}\vphi(\om^{n}-\om_{\vphi}^{n})
=\frac{\i}{V}\sum_{i=0}^{n-1}\int_{M}\p\varphi\wedge\bar\p\vphi\wedge\om^{i}\wedge\om_\vphi^{n-1-i},\\
J_\om(\vphi)&=&\frac{\i}{V}\sum_{i=0}^{n-1}\frac{i+1}{n+1}
\int_{M}\p\vphi\wedge\bar\p\vphi\wedge\om^{i}\wedge\om^{n-1-i}_{\vphi}.
\eeqs Note that the functionals $I$ and $J$  satisfy
the inequalities
\begin{align*}
\frac{1}{n+1}I\leq J\leq\frac{n}{n+1}I.
\end{align*}
In Ding \cite{MR967024}, the Lagrangian functional of the
Monge-Amp\`ere operator is
\begin{align}\label{i functional}
D_\om(\vphi)=\frac{1}{V}\int_M\vphi\om^n-J_\om(\vphi).
\end{align}
The derivative of $D_{\oo}$ along a general path $\varphi_t\in\cH(M, \oo)$
is given by
\begin{align*}
\frac{d}{dt}D_\om(\vphi_t)=\frac{1}{V}\int_{M}\dot\vphi_t\,\om^n_{\vphi_t}.
\end{align*}
We   compute the explicit formula of $D_\om(\vphi)$ as the
following
\begin{align}
D_\om(\vphi)
&=\frac{1}{V}\sum_{i=0}^n\frac{n!}{(i+1)!(n-i)!}
\int_{M}\vphi\om^{n-i}\wedge(\pbp\vphi)^i\nonumber\\
&=\frac{1}{V}\int_M\vphi\om^n-\frac{\i}{V}\sum_{i=0}^{n-1}\frac{i+1}{n+1}
\int_{M}\p\vphi\wedge\bar\p\vphi\wedge\om^{i}\wedge\om_\vphi^{n-i-1}.
\label{eq:D}
\end{align}
Let $\cH_0(M, \oo)$ be the subspace of $\mathcal H(M, \oo)$ with the normalization condition
$$\mathcal H_0(M, \oo):=\{\vphi\in \mathcal H(M, \oo)\vert D_\om(\vphi)=0\}.$$
It is obvious that the Calabi flow stays in $\mathcal H_0(M, \oo)$ once it starts from
a K\"ahler potential in $\mathcal H_0(M, \oo)$. It is proved in Chen \cite{MR1863016} that
any two K\"ahler potentials in $\mathcal H(M, \oo)$ are connected by a $C^{1,1}$ geodesic.
The length of the $C^{1,1}$ geodesic from $0$ to $\vphi$ is given by
$$d_\om(\vphi)=\int_0^1\sqrt{\int_M\,\dot\vphi^2\om_\vphi^n}\,dt.$$
It satisfies the following inequality
\begin{align*}
d_\om(\vphi)\geq V^{-\frac{1}{2}}\max\Big\{\int_{\vphi>0}\vphi\om_\vphi^n,  \;-\int_{\vphi<0}\vphi\om^n\Big\}.
\end{align*}
Recall the explicit formula of Mabuchi's $K$-energy (cf.
\cite{MR1772078}  \cite{MR1787650})
\begin{align}\label{k energy}
\nu_\om(\vphi) &=E_{\oo}(\vphi) +\as D_\om(\vphi)+j_\om(\vphi).
\end{align} The advantage of this formula is that it is well-defined for all bounded
K\"ahler metrics which may be degenerate. The first and the third
terms of (\ref{k energy}) will be discussed in the following. We call
the first functional the entropy of  K\"ahler metrics
\begin{align}
E_{\oo}(\vphi)=\int_M \log\frac{\om^n_\vphi}{\om^n}\,\om^n_\vphi.
\end{align}
The entropy $E_{\oo}(\vphi)$ is uniformly bounded below, since
\begin{align*}
\log\frac{\om_\vphi^n}{\om^n}\frac{\om_\vphi^n}{\om^n}\geq -e^{-1}.
\end{align*}
Applying Tian's $\alpha$-invariant and the Jensen inequality, we obtain a  lower bound of the entropy.
\begin{lem}\label{lem:E and I}
There are uniform constants $\dd=\dd(\oo), C=C(\oo)>0$ such that for
any $\varphi\in \cH(M, \oo)$ we have
\begin{align}
E_\om(\vphi)\geq \dd
I_{\oo}(\varphi) -C.
\end{align}
\end{lem}
The third term of (\ref{k energy}) is the $j$-functional which is
given by
\begin{align}
j_\om(\vphi)
=-\frac{1}{V}\sum_{i=0}^{n-1}\frac{n!}{(i+1)!(n-i-1)!}\int_{M}\vphi
Ric(\oo)\wedge\om^{n-1-i}\wedge(\pbp\vphi)^{i}.\label{eq:j}
\end{align}
Along a path $\varphi_t\in \cH(M, \oo)$, we have
\begin{align*}
\frac{d}{dt}j_\om(\vphi_t)=-\frac{n}{V}\int_{M}\dot\vphi_t
Ric(\oo)\wedge\om^{n-1}_{\vphi_t}.
\end{align*}
Under the topological condition that
 \begin{align}\label{jtop}
C_1(M)<0 \text{ and } -n\frac{C_1(M)\cdot[\oo]^{n-1}}{[\oo]^n}[\oo]+(n-1)C_1(M)>0,
\end{align} the $j$-flow which is the gradient flow of
the $j$-functional converges and
 thus the $j$-functional has lower bound
 (see Chen \cite{MR1772078}, Song-Weinkove \cite{MR2368374}).
 Combining this result with \eqref{k energy}, we get that
 along the Calabi flow the entropy $E$ has upper bound.
When $C_1(M)<0$ and $M$ admits a cscK metric, it is observed by J.
Streets \cite{Streets:2012uq} that along the Calabi flow the lower
bound of $j$ is controlled by the geodesic distance. Since the
geodesic distance to a cscK metric (which is stable along the Calabi
flow) is decreasing, the functional $j$  is bounded below. Thus
along the Calabi flow the entropy $E$ also has upper bound. In
general, in order to obtain the upper bound of the entropy, we
require the ``coercive" condition on the $K$-energy. The properness
of $K$-energy is introduced by Tian to prove the existence of
K\"ahler-Einstein metrics on Fano manifolds in \cite{MR1471884}. The space of K\"ahler metrics equipped with $L^2$-metric introduced by Donaldson
  \cite{MR1736211}, Mabuchi \cite{MR909015} and Semmes
  \cite{MR1165352} is an infinite-dimensional, nonpositive curved,
  symmetric space, its geodesic distance function is a natural positive
  function. Chen defined another properness of the $K$-energy
regarding to the entropy in \cite{MR2471594}  and the geodesic distance $d$ in \cite{MR1772078}. We list them as the following.
\begin{defn}\label{E and I}
 Let $\rho(t)$ be a nonnegative, non-decreasing
functions  satisfying $\lim_{t\rightarrow\infty}\rho(t)=\infty$. The
$K$-energy is called to be:
\begin{itemize}
\item $I$-proper on $\cH(M, \oo)$, if
$ \nu_{\oo}(\vphi)\geq \rho(I_{\oo}(\varphi)) $ for all
$\varphi\in \cH(M, \oo);$
\item $E$-proper on $\cH(M, \oo)$, if
$\nu_{\oo}(\vphi)\geq \rho(E_{\oo}(\varphi)) $ for all
$\varphi\in \cH(M, \oo);$
\item $d$-proper on $\cH(M, \oo)$, if
$\nu_{\oo}(\vphi)\geq \rho(d_{\oo}(\varphi)) $ for all
$\varphi\in \cH(M, \oo).$
\end{itemize}
\end{defn}
By \lemref{E and I}, we see that $E$-properness implies
$I$-properness. The following result says that the converse is also
true in our situation.

\begin{lem}\label{lem:proper}Fix a constant $C$.
Let $\cH_C$ be the set of functions $\varphi\in \cH(M, \oo)$ which satisfies
the inequality \beq \osc(\varphi)\leq I_{\oo}(\varphi)+C.
\label{eq:C}\eeq
 Then the $K$-energy is $E$-proper on $\cH_C$ if and
only if the $K$-energy is $I$-proper on $\cH_C.$
\end{lem}
\begin{proof}That $E$-properness implies $I$-properness follows from
\lemref{lem:E and I}. It suffices to show the converse.
Suppose that the $K$-energy is $I$-proper on $\cH_C$.  From  \eqref{eq:C},
$\int_M\,\varphi\,\oo^n$ is bounded by $I_\om(\vphi)+C$.
So from (\ref{eq:D}) the function $D_{\oo}(\varphi)$ satisfies
$$|D_{\oo}(\varphi)|\leq C_1(I_{\oo}(\varphi)+1),\quad \forall\,\varphi\in \cH_C$$
for some constant $C_1$. On the other hand, each term of
(\ref{eq:j}) can be written as a linear combination of the following
expressions \beq \int_M\; \varphi  Ric(\oo)\wedge \oo^{n-1}, \quad
\int_M\;\varphi Ric(\oo)\wedge \oo^{n-2-k}\wedge
\oo_{\varphi}^k\wedge \pbp \varphi,\quad k=0, \cdots, n-2.
\label{eq:E}\eeq Since $\varphi\in \cH_C$, the first term of
(\ref{eq:E}) can be controlled by $I_{\oo}(\varphi).$ The rest terms
of (\ref{eq:E}) can be estimated as follows: \beqs
&&\Big|\int_M\;\varphi Ric(\oo)\wedge \oo^{n-2-k}\wedge
\oo_{\varphi}^k\wedge \pbp
\varphi\Big|\\
&=& \Big|\int_M\;\i\p\varphi\wedge \bar \p \varphi\wedge
Ric(\oo)\wedge \oo^{n-2-k}\wedge \oo_{\varphi}^k\Big|\\
&\leq &C_2\int_M\;\i\p\varphi\wedge \bar \p \varphi \wedge
\oo^{n-1-k}\wedge \oo_{\varphi}^k\\&\leq& C_3I_{\oo}(\varphi),\eeqs
where we used $Ric(\oo)\leq C_2\oo$ in the first inequality. Combining
the above estimates with (\ref{eq:j}), we have
$$|j_{\oo}(\varphi)|\leq C_4I_{\oo}(\varphi).$$
Thus, by (\ref{k energy}) we have \beq
E_{\oo}(\varphi)=\nu_{\oo}(\varphi)-\un
SD_{\oo}(\varphi)-j_{\oo}(\varphi)\leq \nu_{\oo}(\varphi)
+C_5I_{\oo}(\varphi).\no\eeq
Here all $C_i$ are positive constants which depends only on $C, n$ and
$\oo.$
Therefore, if the $K$-energy is
$I$-proper and is bounded from above, the functional $E_{\oo}$ is
also bounded. Thus, the $K$-energy is $E$-proper. The lemma is
proved.

\end{proof}

In the case of extremal K\"ahler metrics, we can define
the modified $K$-energy by (\ref{mkenergy}) for the $K$-invariant K\"ahler
potentials (see Section \ref{Sec:Calabiflow} for the details).
As in Definition \ref{E and
I}, we can define the $E$ and $I$ properness for the modified
$K$-energy. After slightly modification, we can prove the following
result as Lemma \ref{lem:proper}.

\begin{lem}\label{lem:proper2}Fix a constant $C$.
Let $\cH_C^K$ be the set of $K$-invariant functions $\varphi\in \cH(M, \oo)$ which satisfies
the inequality \beq \osc(\varphi)\leq I_{\oo}(\varphi)+C.
\label{eq:C2}\eeq
 Then the modified $K$-energy is $E$-proper on $\cH_C^K$ if and
only if the modified $K$-energy is $I$-proper on $\cH_C^K.$
\end{lem}
\begin{proof}It suffices to show the sufficiency part. By the
expression of the modified $K$-energy,
$$\td \nu_{\oo}(\varphi)=\nu_{\oo}(\varphi)+\int_0^1\;dt\int_M\; \pd {\varphi_t}t
\te_{\Re V}(\varphi_t)\,\oo_{\varphi_t}^n,$$ where $\varphi_t\in
\cH_C^K$ is a path from $0$ to $\varphi.$ Following the argument of
Lemma \ref{lem:proper}, we only need to show \beq
\Big|\int_0^1\;dt\int_M\; \pd {\varphi_t}t \te_{\Re
V}(\varphi_t)\,\oo_{\varphi_t}^n\Big|\leq C'I_{\oo}(\varphi)
\label{eq:C3} \eeq for some constant $C'=C'(n, \oo)>0.$ Choosing the path $\varphi_t=t\varphi(0\leq t\leq
1)$, we calculate the left hand side of (\ref{eq:C3}) \beq
\int_0^1\;dt\int_M\; \pd {\varphi_t}t \te_{\Re
V}(\varphi_t)\,\oo_{\varphi_t}^n=\int_0^1\;dt\int_M\varphi (\te_{\Re
V}
+t\,X(\varphi))\Big((1-t)\oo_g+t\oo_{\varphi}\Big)^n.\label{eq:C4}\eeq
The right hand side of (\ref{eq:C4}) is a linear combination of the
following expressions: \beq \int_M\;\varphi \te_{\Re V}\oo_g^i\wedge
\oo_{\varphi}^{n-i},\quad \int_M\;\varphi
\,X(\varphi)\,\oo_g^i\wedge \oo_{\varphi}^{n-i},\quad i=0,\cdots, n.
\label{eq:C5} \eeq Since $\varphi$ satisfies the condition
(\ref{eq:C2}) and $X(\varphi)$ is bounded (cf. \cite{MR1314584}\cite{MR1817785}),
all the expressions in (\ref{eq:C5}) are bounded by
$I_{\oo}(\varphi).$ Thus, (\ref{eq:C3}) is proved and the lemma
holds.

\end{proof}

\subsection{Linear estimate}
\begin{lem}\label{lem:sobolev}(Sobolev inequality, Hebey
\cite{MR1385278}) Let  $(M,g)$ be an $m$-dimensional complete
Riemannian manifold with Ricci curvature bounded below and positive
injective radius. For any $\eps>0$ and $1\leq q<m$, there exists
constant $A(q,\eps)$ such that for any $f\in W^{1,q}(M)$,
\begin{align*}
(\frac{1}{V}\int_M|f|^{p}dV_g)^{\frac{1}{p}} \leq
(K(m,p)+\eps)\Big(\frac{1}{V}\int_M|\nabla
f|^qdV_g\Big)^\frac{1}{q}+A(q,\eps)\Big(\frac{1}{V}\int_Mf^qdV_g\Big)^\frac{1}{q},
\end{align*} where the constant $p$ is defined by  $\frac{1}{p}+\frac{1}{m}=\frac{1}{q}.$
\end{lem}
In the following, we call $C_S=\max\{(K(m,p)+\eps),A(q,\eps)\}$ the
Sobolev constant.

\begin{lem}(Poincar\'e inequality, Li-Yau \cite{MR573435}) Let
$(M,g)$ be a $m$-dimensional complete Riemannian manifold with Ricci
curvature bounded below and bounded diameter, then for any $f\in
W^{1,q}(M)$ we have
\begin{align*}
\frac{1}{V}\int_M(f-\frac{1}{V}\int_Mf\,dV_g)^2\,dV_g
\leq\frac{1}{\la_1 V}\int_M|\nabla f|^2\,dV_g.
\end{align*}
\end{lem}
In the following, we call $C_p=\frac{1}{\la_1 V}$ the Poincar\'e
constant. The following result says the Poincar\'e constant is
bounded by the Sobolev constants and the lower bound of the Ricci
curvature.

\begin{lem}\label{lem:poincare}
Let $\cC$ be a set of K\"ahler metrics in the K\"ahler class $\Om$
such that
  the K\"ahler metrics in $\cC$ have uniform Sobolev
constant and uniform lower bound of the Ricci curvature. Then the
K\"ahler metrics in $\cC$ have the uniform Poincar\'e constant.
\end{lem}
\begin{proof}
Since the Sobolev constant is bounded, by Carron
\cite{MR1427759} there are two positive constants $r_0$ and $\kappa$ such that
$$\vol(B(p, r))\geq \kappa r^{2n},\quad \forall \,r\in (0, r_0).$$
Since the volume of $\oo$ is fixed, the diameter is uniformly
bounded from above. This together with the lower boundedness of
Ricci curvature implies that the first eigenvalue of the Laplacian
operator is uniformly bounded from below (cf. Li-Yau
\cite{MR573435}). Thus, the Poincar\'e constant is uniformly
bounded.
\end{proof}

Now we consider the following linear equation on a
compact K\"ahler manifold $(M,\om_\vphi)$ of complex dimension $n$:
\begin{align}\label{linear equ ge}
\tri_\vphi v = f.
\end{align}
Since the metrics $\oo_{\varphi}$ are in the same K\"ahler class, we
assume the volume of $\oo_{\varphi}$ is $1$. The zero order estimate
follows from the De Giorgi-Nash-Moser iteration.
\begin{prop}\label{inte est}(Global boundedness)
If $v$ is a $W^{1,2} $ sub-solution (respectively super-solution) of
\eqref{linear equ ge} in $M$; moreover, if $f\in L^{\frac{p}{2}}$
with $p> 2n$, then there is a constant $p^\ast=\frac{2np}{2n+p}$ and
$C$ depending on the Sobolev constant $C_S$ and the first eigenvalue
$\la_1$ with respect to $\oo_{\vphi}$ such that
\begin{align}\label{glob bound claim}
\sup_M \Big(v-\int_M v\om^n_\vphi\Big) \leq C\| f \|_{p^\ast} \qquad
\Big( \hbox{resp.} \; \sup_M \Big(v-\int_M v\om^n_\vphi \Big) \leq
C\| f \|_{p^\ast}\Big).
 \end{align}
\end{prop}
\begin{proof}
Let $\tilde v=v-\int_M v\,\om^n_\vphi$. For any constant $k$ we denote by $u=(\tilde
v-k)_+$  the positive part of  $\tilde v-k$.
Set $A(k)=\{x\in M\vert \tilde v(x)>k\}$. Multiplying \eqref{linear
equ ge} with $u$ on both  sides and integrating by parts, we have the inequality
$$\int_M |\Na u|^2\om_\vphi^n\leq -\int_M u f \om_\vphi^n.$$
By the H\"older's inequality, the right hand side is bounded by
\begin{align*}
\|u\|_{2^\ast}\cdot \|f\|_{p^\ast}\cdot |A(k)|^r,
\end{align*} where $\|\cdot\|_p$ denotes the $L^p$ norm with respect to the metric $\oo_{\varphi}$ and $|A(k)|$ denotes the volume of $A(k)$ with respect to $\oo_{\varphi}$.
Here $m=2n$, $2^\ast=\frac{2m}{m-2}$,  $p^\ast=\frac{mp}{m+p}$ and $r=\frac{1}{2}-\frac{1}{p}$.
The Sobolev inequality implies
$$\|u\|^2_{2^\ast}\leq C_1(\|\Na u\|^2_2+\|u\|^2_2).$$
Here all constants $C_i$ in the proof   depend  on
$C_S(\oo_{\varphi}).$ Since $\|u\|_2\leq C_2\|u\|_{2^\ast}
|A(k)|^\frac{1}{m}$, we obtain \beq \|u\|^2_{2^\ast}\leq
C_3(\|u\|_{2^\ast}\cdot\|f\|_{p^\ast}\cdot |A(k)|^r+\|u\|^2_{2^\ast}\cdot
|A(k)|^\frac{2}{m}).\label{eqa2}\eeq We will choose a constant $k_0$ later
such that for any $k\geq k_0$ we have $|A(k)|^\frac{2}{m}\leq
\frac{1}{2C_3}$. Using  the Poincar\'e inequality, we have \beq
\|\tilde v\|_2^2\leq C_4\|\Na \tilde v\|_2^2 .\label{eqa1} \eeq
By multiplying the
equation \eqref{linear equ ge} with $\tilde v$ and applying the
H\"older inequality, we have the inequality   $\|\Na\tilde v\|_2^2\leq \|\tilde
v\|_2\|f\|_2$. Thus we have
 $$\|\tilde v\|_2 \leq C_4\|f\|_2\leq C_5\|f\|_{p^\ast}.$$ Since $k_0^2 |A(k_0)|\leq
\|\tilde v\|_2^2$, we choose $k_0^2=C_5\|f\|_{p^\ast}
(2C_3)^{\frac{m}{2}}$ so that we have  $|A(k_0)|^\frac{2}{m}=\frac{1}{2C_3}.$
Combining this with (\ref{eqa2}), we  obtain for any $k\geq k_0$,
$$\|u\|_{2^\ast}\leq C_6\|f\|_{p^\ast}\cdot |A(k)|^{r}.$$
Now we choose $h>k\geq k_0$ so that we have the inequality
$$\|u\|_{2^\ast}\geq C_7(h-k)\cdot |A(h)|^{\frac{1}{2^\ast}}.$$
Thus by the iteration lemma (see \cite{MR1814364}),
we have $A(k_0+d)=0$ and $d=C_8 \|f\|_{p^\ast}$. Thus, we have
$$\tilde v\leq k_0 + d\leq C_9
\|f\|_{p^\ast}.$$ The proposition is proved.
\end{proof}

\begin{rem}
Since $p^\ast=\frac {2np}{2n+p}$, we have  $n<
p^\ast<\frac{p}{2}$.
\end{rem}

\subsection{Log volume ratio}

\begin{lem}\label{est h}The log volume ratio $h_{\oo}(\varphi):=\log\frac{\om^n_\vphi}{\om^n}$
satisfies the equality \beq |h_{\oo}(\varphi)|\leq
E_{\oo}(\varphi)+ C(n, p,
C_S(\oo_{\varphi}))\Big(\osc_M\varphi+\|S(\oo_{\varphi})\|_{L^p(\oo_{\varphi})}\Big),
\eeq where $p>n.$
\end{lem}
\begin{proof}For brevity, we write $h$ for
$h_{\oo}(\varphi)$. We compute the equation of the scalar curvature,
\begin{align*}
\tri_\vphi h= g_\vphi^{i\bar j}R_{i\bar j}(\om)-S(\oo_{\varphi}).
\end{align*}
Since the background metric $\om$ is   smooth, we have   $$\inf_M Ric \cdot \om\leq Ric(\om)\leq \sup_M Ric \cdot \om.$$
Letting $R_-=-\inf_M Ric, R_+=\sup_M Ric$ be two positive constants, we have
\begin{align*}
-R_- \cdot(n-\tri_\vphi\vphi) -S(\oo_\varphi)\leq\tri_\vphi h\leq
R_+\cdot (n-\tri_\vphi\vphi) -S(\oo_\varphi).
\end{align*}
Thus we have the inequality \beqn  \tri_\vphi (h-R_-\cdot
\vphi)&\geq& -nR_- -S(\oo_\varphi), \label{lmaupp}\\
\tri_\vphi (h+R_+\cdot \vphi)&\leq& nR_+ -S(\oo_\varphi).
\label{lmalow} \eeqn By Proposition \ref{inte est} and the
inequalities \eqref{lmaupp}-(\ref{lmalow}),  we have \beqs
h-E_{\oo}(\vphi)&\leq&
C\Big(R_-(\vphi-\int_M\vphi\om_\vphi^n)+\|S(\oo_\varphi)\|_{L^p(\oo_{\varphi})}\Big)\\
&\leq&
C\Big(R_-(\osc_M\vphi)+\|S(\oo_\varphi)\|_{L^p(\oo_{\varphi})}\Big)
\eeqs and
\beqs
h-E_{\oo}(\vphi)&\geq&
-C\Big(R_+(\vphi-\int_M\vphi\om_\vphi^n)+\|S(\oo_\varphi)\|_{L^p(\oo_{\varphi})}\Big)\\
&\geq&
-C\Big(R_+(\osc_M\vphi)+\|S(\oo_\varphi)\|_{L^p(\oo_{\varphi})}\Big).
\eeqs Thus, the lemma is proved. \end{proof}

\begin{rem}The lower bound of $h$ is obtained by Chen-Tian \cite{MR1893004}
 under the Ricci curvature lower bound. If $Ric_{\vphi}$ is bounded from below
 for some constant $L$, i.e.
$$Ric_{\vphi}\geq-\Lambda\om_\vphi,$$
then there exists a constant $C(-\Lambda+\frac{1}{n}\sup_MS(\om))$ such that
\begin{align*}
\inf_M{h_{\oo}(\varphi)}&\geq-4\Big(L\sup_M(\vphi-\int_M\vphi\om^n)+C\Big)
e^{2(1+E_{\oo}(\vphi))}.
\end{align*}

\end{rem}
\subsection{Second order estimate}
The following result is a standard application of Chern-Lu
inequality:
\begin{lem}\label{2ndestimat}There is a constant $C>0$ depending on $\inf_MRic(\oo_{\varphi}),
\sup_{i, j}R_{i\bar ij\bar j}(\oo), \osc_M(\varphi)$ and $\sup_Mh_{\oo}(\varphi)$
such that
$$\frac 1C\oo\leq \oo_{\varphi}\leq C\oo.$$
\end{lem}
\begin{proof} As in
\cite{MR0486659}\cite{MR0250243}, we have the Chern-Lu inequality:
\begin{align*}
\tri_\vphi\log \tr_{\om_\vphi}\om
&=\frac{\tri_\vphi(\tr_{\om_\vphi}\om)}{\tr_{\om_\vphi}\om}-\frac{g_\vphi^{k\bar
l}\cdot \p_k{g}_\vphi^{i\bar j}g_{i\bar j}\cdot\p_{\bar
l}{g}_\vphi^{p\bar q}
g_{p\bar q}}{(\tr_{\om_\vphi}\om)^2}\\
&\geq\frac{R_\vphi^{i\bar j}g_{i\bar j}-g_\vphi^{i\bar
j}g_\vphi^{k\bar l}R_{i\bar jk\bar
l}}{\tr_{\om_\vphi}\om}\\&\geq\inf_M Ric(\oo_{\varphi})-\sup_{i,
j}R_{i\bar ij\bar j}(\oo)\, \tr_{\om_\vphi}\om.
\end{align*}
Let $\la=\sup_MR_{i\bar ij\bar j}(\oo)+1.$ The function
$u:=\log\tr_{\oo_{\varphi}}\oo-\la \varphi$ satisfies the inequality
\beqs \Delta_{\varphi}u&\geq& (\inf_MRic(\oo_{\varphi})-\la
n)+(\la-\sup_MR_{i\bar ij\bar j}(\oo))\tr_{\oo_{\varphi}}\oo\\&\geq
& (\inf_MRic(\oo_{\varphi})-\la n)+\tr_{\oo_{\varphi}}\oo.\eeqs
 By the maximum principle, we have $$\tr_{\om_{\varphi}}\om\leq
C(\inf_M Ric(\oo_{\varphi}), \sup_{i, j}R_{i\bar ij\bar j}(\oo),
\osc_M(\varphi)).$$ This implies that $\oo_{\varphi} \geq \frac
1C\oo$. On the other hand, using the identity
$\oo_{\varphi}^n=e^h\oo^n$ we have
$$\oo_{\varphi}\leq  C'\oo$$
for some positive constant $C'=C'(\inf_M Ric(\oo_{\varphi}),
\sup_{i, j}R_{i\bar ij\bar j}(\oo), \sup_Mh, \osc_M(\varphi)).$ The
lemma is proved.

\end{proof}

\subsection{Zero order estimate}

\begin{lem}\label{lem:zero}For any K\"ahler potential $\varphi\in \cH(M, \oo),$ we have
\beq \osc(\varphi)\leq I(\oo, \oo_{\varphi})+C(C_S(\oo_{\varphi}),
C_S(\oo)). \label{eq403}\eeq
\end{lem}
\begin{proof}
Since $ n+\tri\vphi>0$, by Proposition \ref{inte est} we have the
upper bound of $\vphi$,
\begin{align}\label{vphiupp}
\sup_M\vphi-\int_M\vphi\,\om^n\leq C(C_S(\oo)).
\end{align} On the other hand, using the inequality $n-\tri_\vphi\vphi>0$
 we have the lower bound
\begin{align}\label{vphilow}
\inf_M\vphi-\int_M\vphi\,\om_\vphi^n\geq -C(\oo_{\varphi}).
\end{align}
The inequality (\ref{eq403}) follows directly from   \eqref{vphiupp} and \eqref{vphilow}.

\end{proof}

Combining Lemma \ref{lem:zero} with Lemma \ref{lem:proper} ( or
Lemma \ref{lem:proper2}), we have

\begin{lem}\label{lem:C0}
Suppose that the (modified) $K$-energy is $I$-proper.
If the (modified) $K$-energy $\nu_{\oo}(\varphi)$(or $\td
\nu_{\oo}(\varphi)$) and $C_s(\oo_{\varphi})$ are bounded, then
$\osc(\varphi)$ and $E_{\oo}(\varphi)$ are bounded.
\end{lem}

\subsection{Proof of Theorem \ref{compactness}}
\begin{proof}[Proof of Theorem \ref{compactness}] We consider the equation
\begin{align}\label{volume ratio equation}
\oo_{\varphi}^n=e^{h_{\oo}(\varphi)}\oo^n.
\end{align}
Under the assumptions of Theorem \ref{compactness}, by Lemma
\ref{lem:C0} $\osc(\varphi)$ and $E_{\oo}(\varphi)$ are bounded and
by Lemma \ref{est h} $h_{\oo}(\varphi)$ is bounded.  So $\vphi$ has
$C^\alpha$ bound by Kolodziej \cite{MR2425147}. Moreover by Lemma
\ref{2ndestimat}, we obtain the second order estimates. Now since the
metric $\om_\vphi$ and $\om$ are equivalent, we apply weak Harnack
inequality and the local boundedness estimate to \eqref{lmaupp} and
\eqref{lmalow} respectively (see Section 5.2 in Calamai-Zheng
\cite{Calamai:2012uq}), we obtain that $h$ has $C^\alpha$ bound. So
the $C^{2,\alpha}$ estimate of $\vphi$ follows from the Schauder
estimate of complex Monge-Amp\'ere equation (see Wang
\cite{MR3008426}). Now we run the bootstrap argument.
 Considering \eqref{volume ratio equation}, the right hand side
 belong to $L^p$ for some $p\geq n$. Since the coefficient is $C^\alpha$,
 we could apply the $L^p$ theory and obtain that $h$ belongs to $W^{2,p}$.
  Consequently, by the Sobolev imbedding theorem $h$ is $C^{1,\alpha}$,
  then $\vphi$ is $C^{3,\alpha}$ by linearizing \eqref{volume ratio equation}
  and applying the Schauder estimate of the Laplacian equation. The theorem is proved.

\end{proof}
The following version of \thmref{compactness} is used to the modified Calabi flow in the remainder sections.
\begin{theo}\label{compactness extremal}Let $(M, \Om)$ be
a compact K\"ahler manifold for which the modified $K$-energy is
$I$-proper in the K\"ahler class $\Om$. If $\cS$ is
the  set of $K$-invariant K\"ahler metrics in $\Om$ satisfying the following
properties:
\begin{itemize}
\item the modified $K$-energy is bounded;
\item the $L^p$-norm of the scalar curvature is bounded for some $p>n$;
\item the Sobolev constant is bounded;
\item the Ricci curvature is bounded from below;
\end{itemize}
then $\cS$ is compact in $C^{1,\alpha}$-topology of the space of the
K\"ahler metrics for some $\alpha\in(0,1)$. In particular, K\"ahler collapsing does not
occur.\\
\end{theo}

\section{Calabi flow}\label{Calabi flow}\label{Sec:Calabiflow}
In this section, we will recall some basic facts of Calabi flow and modified Calabi flow, which will be used in next sections. Let $(M, \oo)$ be a compact K\"ahler manifold with a K\"ahler metric $\oo$. A family of K\"ahler potentials $\varphi(t)\in \cH(M, \oo)(t\in [0, T])$ is called a solution of Calabi flow, if it satisfies the equation
$$\pd {}t\varphi(t)=S(\varphi(t))-\un S,$$
where $S(\varphi(t))$ is the scalar curvature of the metric $\oo_{\varphi(t)}.$ In the level of K\"ahler metrics, the equation of Calabi flow can be written as
$$\pd {}tg_{i\bar j}=S_{i\bar j}, $$
where $g_{i\bar j}$ denotes the metric tensor of the K\"ahler form $\oo_{\varphi(t)}.$
Along the Calabi flow,  the evolution equation of the curvature tensor is
given by
\begin{align}\label{evorm}
\frac{\partial}{\partial t} Rm=-\Na\bar \Na\Na\bar \Na S
=-\Delta^2Rm+\Na^2Rm*Rm+\Na Rm*\Na Rm,
\end{align} where the operator $*$ denotes some contractions of
tensors.

Next, we introduce the modified Calabi flow.
Let $\Aut(M)$ be the holomorphic group and $\Aut_0(M)$  the
identity component
 of the holomorphic group. According to Fujiki \cite{MR0481142}, $\Aut_0(M)$
 has a unique meromorphic subgroup
$G$ which is a linear algebraic group and the quotient $\Aut_0(M)
/G$ is a complex torus. Furthermore, the Chevalley
 decomposition shows that $G$ is the semidirect product of
 reductive subgroup $H$ and the unipotent radical $R_u$. Moreover,
  $H$ is a complexification of a maximal compact subgroup $K$ of $G/ R_u$.
  Let $\frak h_0(M)$, $\mathfrak g$ and $\mathfrak h$ be
  the Lie algebra of $\Aut_0(M)$, $G$ and $H$ respectively.

Let $\mathcal M_{{\inv}}$ contains the metric in K\"ahler class
$\Omega$ whose isometric group is a maximal compact subgroup $K$ of
$\Aut_0(M)$. We call these metrics the $K$-invariant K\"ahler
metrics. Each  metric $\omega\in \mathcal M_{\inv}$ defines a
holomorphic vector field
 $$V_\omega=g^{i\bar j}\frac{\partial (\pr S(\omega))}{\partial z^i} \frac{\partial}{\partial z^{\bar j}},$$
where $\pr$ is the projection from the space of complex valued
functions to the space of the Hamiltonian functions of the
holomorphic vector fields. Due to Calabi \cite{MR780039} we have the isometric group
of an extremal metric is a maximal compact subgroup of $\Aut_0(M)$
and the associated $V$ is a holomorphic
vector field. Hence, all extremal metrics stay in $\mathcal M_{\inv}$.
 Futaki-Mabuchi \cite{MR1270938} showed that $\Re V$ lies in the center
 of $\mathfrak h$ and is independent of the choice
 of $\omega\in \mathcal M_{\inv}$. Moreover, for any $\omega_1,\omega_2$,
  there exists a element $\sigma$ in the unipotent
  radical $R_u$ such that $\sigma_\ast V_{\omega_1}= V_{\omega_2}$.
Let $K$ be a maximal compact subgroup which contains
isometric group of $\omega$ generated by $\Im V$.
If $\omega$ is a $K$-invariant metric, then there is a real-valued
Hamiltonian function $\theta_{\Re V}$ such that
\begin{align*}
L_{\Re V}\omega=\sqrt{-1}\partial\bar\partial\theta_{\Re V}.
\end{align*}
Let $\mathcal H_K(M, \oo)$ be the space of $K$-invariant K\"ahler potentials
\begin{align}
\mathcal H_K(M, \oo)
&=\{\varphi\in C_R^{\infty}(M)\,\vert\,\omega
+\sqrt{-1}\partial\bar\partial\varphi>0,
L_{\Im V}(\omega_\varphi)=0\}.
\end{align}
If  $\sigma(t)$ is  the holomorphic group
generated by the real part of $V$. i.e. $\Re V
=\sigma^{-1}_\ast(\frac{\partial}{\partial t}\sigma),$ then there is a smooth function
$\rho(t)\in \cH(M, \oo)$ such that  $$\sigma(t)^{\ast}\omega
=\omega+\sqrt{-1}\partial\bar\partial\rho(t).$$
Let $\varphi(t)$ be the solution of Calabi flow and $\psi(t)=\sigma(-t)^\ast(\varphi(t)-\rho(t))$. Then $\sigma(t)^\ast\om_\psi=\om_\vphi$. When the initial K\"ahler potential is $K$-invariant, the solution $\varphi(t)$ remains to be $K$-invariant.
Moreover, we have
\begin{align*}
L_{\Re V}\omega_t
=\sqrt{-1}\partial\bar\partial\theta_{\Re V}(t)
=\sqrt{-1}\partial\bar\partial(\theta_{\Re V}+{\Re V}(\psi))
\end{align*}

The function $\psi(t)$ is called a solution of modified Calabi flow, which is given by
the equation
\begin{align}\label{mcf}
\frac{\partial \psi}{\partial t}
=S(\oo_\psi)-\underline S-\theta_{\Re V}(\psi).
\end{align}
The proof of (\ref{mcf}) can be found in  Huang-Zheng
\cite{MR2953047}.  We call $S-\underline{S}-\theta_{\Re V}$ the modified scalar
curvature of $\omega$. The modified Calabi-Futaki invariant is defined in $\mathcal H(M, \oo)$
\begin{align*}
\td F(Y) =-\int_M\theta_{X} (S-\underline{S}-\theta_{V})\,\omega^n.
\end{align*}
It equals to zero when $M$ admits extremal metric $\omega$
and we have $\underline{S}+\theta_{\Re V}=S$. Particularly, if it is restricted in $\mathcal H_K(M, \oo)$, the modified Calabi-Futaki invariant becomes
\begin{align}\label{mkfutaki}
\td F(Y) =-\int_M\theta_{X} (S-\underline{S}-\theta_{\Re
V})\omega^n.
\end{align}
The modified $K$-energy is well-defined for $K$-invariant K\"ahler metrics
\begin{align}\label{mkenergy}
\td \nu(\varphi)& =-\int_0^1\int_M\dot\varphi
(S-\underline{S}-\theta_{\Re V}(\varphi))\,\omega_\varphi^n\,dt.
\end{align}
We can define the modified Calabi energy  by
\begin{align}\label{mcalabie}
\td Ca(\omega) =\int_M(S-\underline{S} -\theta_{\Re V})^2\omega^n.
\end{align}
We denote the
operator  $L$   by
 $$L(u):=u_{,\bar i\bar jji}=\Delta^2u+R_{i\bar j}u_{j\bar i}+S_{, i}u_{\bar i}.$$
 Then the evolution of the modified Calabi energy can be written by
 \begin{align}\label{epo}
\partial_t\int_M\dot\psi^2\omega_\psi^n
=-2\int_M\dot\psi L\dot\psi\omega_\psi^n,
\end{align} and the proof of (\ref{epo}) can be found in Section 3 in Huang-Zheng \cite{MR2953047}.

\section{Estimates along Calabi flow}\label{Eigenvalues estimates}

\subsection{Short time estimates}
 The Calabi flow is a fourth
order quasi-linear parabolic equation of K\"ahler potentials. When the initial data is a smooth K\"ahler potential, the short time existence of the Calabi flow follows from the theory of the quasi-linear parabolic equation. However, the proof of local existence with the H\"older continuous initial
   data is quite different from the case of smooth initial data.
   Da Prato-Grisvard \cite{MR551075}, Angenent \cite{MR1059647} and others developed the abstract theory of local
     existence for the fully nonlinear parabolic equation. They  \cite{MR551075}
     constructed the continuous interpolation spaces so that
     the linearized operator stays in certain classes.
     This theory has been applied to prove the well-posedness
     of the Calabi flow with $C^{3,\alpha}$ initial potentials
     and the pseudo-Calabi flow with $C^{2,\alpha}$ initial potentials
     (see Theorem 3.2 in Chen-He \cite{MR2405167} and Theorem 4.1 and
     Proposition 4.42 in Chen-Zheng \cite{MR3010550}.  We accumulate
     the results as following.
\begin{prop}\label{short time}Suppose that the initial K\"ahler potential $\varphi(0)$
satisfies that for  integer $l\geq 3$ and positive constant
$\lambda$,
$$\omega_{\varphi(0)}\geq\lambda\omega,\quad |\varphi(0)|_{C^{l, \alpha}}\leq A,
$$ then the Cauchy problem for the Calabi flow
has a unique solution within the maximal existence time $T$.
Moreover, there exists a time $t_0=t_0(\lambda, A, n, \omega)$ such
that for any $t\in [0, t_0]$ we have
$$ \omega_t \geq \frac{\lambda}{2}\omega,\quad |\varphi(t)|_{C^{l, \alpha}}\leq 2A.$$
For any $t_0> \epsilon>0$ and any $k\geq l+1$, there also exists
a constant $C(\lambda, A, k, \epsilon, n, \omega)$, such that
$$ |\varphi(t)|_{C^{k, \alpha}}\leq C,\quad  \forall \;t\in [\epsilon, t_0].$$
\end{prop}

\subsection{The first eigenvalue estimate along the Calabi flow}

In this subsection, we will estimate the first eigenvalue of the
operator  $L_t$ along the Calabi flow when $M$ has no non-zero holomorphic vector fields.  Along the Ricci flow, the
study of the eigenvalues of Laplacian and some other operators has been carried out by
Cao \cite{MR2262792}, Li \cite{MR2317755}, Perelman
\cite{Perelman:zr}, Wu-Wang-Zheng \cite{Wu:2009fk} etc.

\begin{prop}\label{lem:eigen2}Let $\mu_1(t)$ be the first eigenvalue of the
$L$ operator along the Calabi flow. For any constants $\La,  \dd>0$
and positive function $\ee(t)$, if the solution of the Calabi flow
satisfies
 \beq |Ric|(t)\leq \La,
\quad |\Na\bar\Na S|(t)\leq \ee(t),\quad
\forall\, t\in [0, T]\eeq
then we have
$$(\mu_1(t)+\La^2)\geq (\mu_1(0)+\La^2)e^{-26\int_0^t\,\ee(t)\,dt},\quad \forall\,t\in [0, T].$$
\end{prop}
\begin{proof}For any $t\in [0, T]$, consider the smooth functions with the normalization condition
$
\int_M\;f(x, t)^2\,\oo_t^n=1.
$
We define the nonnegative auxiliary function
$$\mu(f, t)=\int_M\; f(x, t)L_tf(x, t)\,\oo_t^n=\|\nabla\nabla f\|_2^2.$$
It is easy to see that $\mu(f, t)\geq \mu_1(t)$ for all $t\in[0,T]$ and the equality holds if and only if $f$ is the eigenfunction of the first eigenvalue.
We calculate the derivative at $t$.
\beqn \frac d{dt}\mu(f,
t)
&=&2\int_M\frac{d}{dt}f L_tf \omega_t^n
+\int_M\,f\Big(\pd{}tL_t\Big)f\,\oo_t^n+
\int_Mf L_tf \Delta_t S\omega_t^n\nonumber.\label{eq001}\eeqn

At $t=t_0$, let $f(x, t_0)$ be the $i$-th eigenfunction of $L_{t_0}$ with
respect to the eigenvalue $\mu(f,t_0)$ satisfying
\begin{align}\label{eigenLin}
L_{t_0}f(x,
t_0)=\mu(f,t_0)f(x, t_0).
\end{align} Using the normalization condition, we have
\beqn \frac d{dt}\mu(f,
t_0)
&=&\int_M\,f\Big(\pd{}tL_t\Big)f\,\oo_t^n\nonumber\\
&=&\int_M\,f\Big(-S_{, m\bar n}f_{\bar m\bar kkn}-(f_{\bar m}S_{,
m\bar k})_{\bar llk}-(f_{, \bar m\bar n}S_{, m\bar pp})_{, n}\Big)\,\oo_t^n\nonumber\\
&:=&(I_1+I_2+I_3). \label{eq001}\eeqn

\begin{lem}
The following inequality holds.
\begin{align}\label{eq004}
\int_M\;|f_{\bar m\bar kk}|^2\leq(2\mu+\La^2)\int_M\;|\Na f|^2.
\end{align}
\end{lem}
\begin{proof}
We have by the Ricci identity,
\beqn \int_M\;|f_{\bar m\bar kk}|^2&=&\int_M\;f_{\bar m\bar
kk}(f_{l\bar
lm}+R_{m\bar n}f_n)\nonumber\\
&=&\int_M\;-f_{\bar m\bar kkm}f_{l\bar l}+f_{\bar m\bar kk} R_{m\bar
n}f_n\nonumber\\
&=&\int_M\;-\mu\cdot f\Delta_t f+f_{\bar m\bar kk} R_{m\bar
n}f_n\nonumber\\&\leq&\int_M\;\mu|\Na f|^2+\frac 12\int_M\;|f_{\bar m\bar
kk}|^2+\frac {\La^2}2\int_M\;|\Na f|^2.\nonumber
\eeqn
Thus, we have
\beqs
\int_M\;|f_{\bar m\bar kk}|^2&\leq&(2\mu+\La^2)\int_M\;|\Na f|^2.
\eeqs

\end{proof}
\begin{lem}
The following inequality holds.
\begin{align}\label{eq0044}
\int_M\;|f_{\bar m\bar nm}|^2
\leq(2\mu+\La^2)\int_M\;|\Na f|^2.
\end{align}
\end{lem}
\begin{proof}
We have by the Ricci identity,
\beqn \int_M\;|f_{\bar m\bar nm}|^2
&=&\int_M\;f_{\bar m\bar nm}((\Delta f)_{n}+R_{n\bar p}f_p)\nonumber\\
&=&\int_M\;-f_{\bar m\bar nm n}\Delta f+f_{\bar m\bar nm} R_{n\bar p}f_p\nonumber\\
&\leq&-\int_M\;\mu \cdot f\Delta_t f+\frac 12\int_M\;|f_{\bar m\bar nm}|^2
+\frac {\La^2}{2}\int_M\;|\Na f|^2\no\\
&=&(\mu+\frac {\La^2}{2})\int_M\; |\Na f|^2+\frac 12\int_M\;|f_{\bar m\bar
nm}|^2.\nonumber
\eeqn
Thus, (\ref{eq0044}) is proved.

\end{proof}
\begin{lem}The   following inequalities hold:
\beqs
\|\Na f\|^2_2&<& 4(\mu^{\frac 12}+\La),\\
\max\{\|f_{\bar m\bar
kk}\|_2^2, \|f_{\bar m\bar nm}\|_2^2\}&<&8(\mu^{\frac 12}+\La)^3.
\eeqs where $C$ is a universal constant.
\end{lem}
\begin{proof}
Applying the Cauchy-Schwarz inequality, we obtain
\beqn
\mu(f, t)&=&\int_M\;(\Delta f)^2+R_{i\bar j}f_{j\bar
i}f+S_{, m}f_{\bar m}f\nonumber\\
&=&\int_M\;(\Delta f)^2+R_{i\bar j}f_{j\bar
i}f-S(f\Delta f+|\Na f|^2)\\
&\geq &\|\Delta
f\|_2^2-\La\|\Na\bar\Na f\|_2-\La\|\Delta f\|_2-\La \|\Na f\|_2^2\nonumber\\
&\geq &\frac 12\|\Delta
f\|_2^2-5\La^2,\nonumber \label{eq003}\eeqn
where we used the Cauchy-Schwarz inequality in the last two
inequalities. Thus, we have
$$\|\Na f\|^2_2\leq \|\Delta f\|_2<4(\mu^{\frac 12}+\La).$$
 So we have from \eqref{eq004} and \eqref{eq0044} \beqn
 \max\Big\{\int_M\;|f_{\bar m\bar
kk}|^2, \int_M\;|f_{\bar m\bar nm}|^2\Big\}&<&8(\mu^{\frac 12}+\La)^3. \nonumber\eeqn
 The lemma
is proved.\\
\end{proof}

Now we estimate $I_1$ at
time  $t$: \beqn I_1&=&\int_M\; -fS_{, m\bar n}f_{\bar m\bar
kkn}\,\oo_{t}^n=\int_M\; \Big(f_nS_{, m\bar
n}f_{\bar m\bar kk}+f(\Delta S)_{, m}f_{\bar m\bar kk}\Big)\,\oo_{t}^n\nonumber\\
&=&\int_M\;\Big( f_nS_{, m\bar n}f_{\bar m\bar kk}-f_m \Delta S
f_{\bar
m\bar kk}-f \Delta S f_{\bar m\bar kkm}\Big)\,\oo_t^n\nonumber\\
&=&\int_M\; \Big(f_nS_{, m\bar n}f_{\bar m\bar kk}-f_m \Delta S f_{\bar
m\bar kk}-\mu\cdot f^2 \Delta S\Big)\,\oo_t^n. \label{eq002}\eeqn
Thus, from \eqref{eq004} we have $I_1:$ \beqn
|I_1|&\leq &2\ee(t)\cdot\|\Na f\|_2
\|f_{\bar m\bar kk}\|_2+\ee(t)\mu\nonumber.  \eeqn
Similarly, we can estimate $I_2$  as follows: \beqn
|I_2|&=&\Big|\int_M\;f_{\bar m}S_{, m\bar k}f_{kl\bar l}\Big|
\leq \ee(t)
\|\Na f\|_2\|f_{\bar m\bar
kk}\|_2
\nonumber.
\label{eq007}\eeqn
Now we estimate $I_3$.
\begin{align*}
I_3&=\int_M\; -f(f_{ \bar m\bar n}R_{, m\bar pp})_{,
n}\,\oo_t^n=\int_M\; f_{
n} f_{ \bar m\bar n}R_{, m\bar pp}\oo_t^n
 =\int_M\; f_{
n} f_{,\bar m\bar n}(\Delta S)_{, m}\oo_t^n\\
&=-\int_M\; f_{
n} f_{ \bar m\bar nm}\Delta S\oo_t^n
-\int_M\; f_{
nm} f_{ \bar m\bar n}\Delta S\oo_t^n.
\end{align*}
So we have
\begin{align*}
|I_3|&\leq\ee(t)\,\|\Na f\|_2\|f_{ \bar m\bar nm}\|_2
+\ee(t)\mu.
\end{align*}
Combining the above inequalities, we have the differential inequality
\beq \frac d{dt}\mu(f, t_0)> -26\ee(t) ( \mu+\La^2) \label{eq400}\eeq
holds. \\

Since $\mu(f, t)$ is a smooth function of $t$, there is a small
neighborhood $[t_0-\dd, t_0+\dd]$ of $t_0$ such that for all $t\in
[t_0-\dd, t_0+\dd]$ the inequality (\ref{eq400}) holds.  Therefore, for any $t_1\in (t_0-\dd, t_0)$ we have
$$\log(\mu(f, t_0)+\La^2)\geq \log(\mu(f, t_1)+\La^2)-26\int_{t_1}^{t_0}\,\ee(t)\,dt.$$
Choose $\mu(f, t_0)=\mu_1(t_0)$, while $ \mu(f, t_1)\geq \mu_1(t_1)$,
we get the inequality \beq \log(\mu_1(t_0)+\La^2)\geq
\log(\mu_1(t_1)+\La^2)-26\int_{t_1}^{t_0}\,\ee(t)\,dt.\label{eq204}\eeq
Since $t_0$ is arbitrary, the inequality (\ref{eq204}) holds for any
$0\leq t_1\leq t_0\leq T.$
So we obtain
$$(\mu_1(t)+\La^2)\geq (\mu_1(0)+\La^2)e^{-26\int_0^t\,\ee(t)\,dt},\quad \forall\,t\in [0, T].$$
The proposition is proved.

\end{proof}

\subsection{The first eigenvalue estimate with bounded K\"ahler potentials}
In this subsection, we will show the exponential decay of the Calabi energy on a
given time interval where the K\"ahler potential is bounded along the Calabi flow. First, we estimate
the eigenvalue of the operator $L$.

\begin{lem}\label{lem:key1} Let $\oo_\varphi=\oo+\pbp\varphi$ be a
K\"ahler metric with the K\"ahler potential in the space
$$\varphi\in \cC_A:=\Big\{\varphi\in C^{\infty}(M, \RR)\;\Big|\;
\oo+\pbp\varphi>0,\;|\varphi|_{C^{4, \al}}\leq A\Big\}.$$ Then there
exists a constant $\mu=\mu(\oo, A)$ such that for any smooth
function $f(x)\in W^{2,2}(M,\omega)$ with
\begin{align}\label{lem:key1con}
\int_M\;f\,\oo_{\varphi}^n=0,\quad
\int_M\;X(h)\,\oo_{\varphi}^n=0,\quad \forall\,X\in \frak
h_0(M),
\end{align} where $h$ is a function determined by the
equation
$\Delta_{\varphi}h=f$ and the normalization condition
$\int_M\; h\,\oo_{\varphi}^n=0,$ the following inequality holds: \beq
\int_M\;|\Na\Na f|^2\,\oo_{\varphi}^n\geq \mu
\int_M\;f^2\,\oo_{\varphi}^n. \nonumber\eeq

\end{lem}
\begin{proof} Without lose of generality, we assume $f$ is a smooth function. We prove
the lemma by the contradiction argument. Suppose the conclusion
fails, we can find a sequence of functions $\varphi_m\in \cC_A$ and
functions $f_m\in C^{\infty}(M)$ satisfying the properties \beq
\int_M\;|f_m|^2\,\oo_{\varphi_m}^n=1,\quad
\int_M\;f_m\,\oo_{\varphi_m}^n=0,\quad
\int_M\;X(h_m)\,\oo_{\varphi_m}^n=0,\quad \forall\,X\in \frak
h_0(M), \label{eq0022}\eeq where $h_m$ is a solution of the equation
$\Delta_{m}h_m=f_m$ with $\int_M\; h_m\;\oo_{\varphi_m}^n=0$ and
\beq \int_M\;|\nabla\nabla f|^2\,\oo_{\varphi_m}^n=\mu_m\ri0. \eeq
Here $\Delta_m$ denotes the Laplacian operator of the metric
$\oo_{\varphi_m}.$ Since $\varphi_m\in \cC_A$, we can assume that
$\varphi_m\ri \varphi_{\infty}$ in $C^{4,\al'}(M)$ for some
$\al'<\al$. Note that \beqs
\int_M\;(\Delta_{m}f_m)^2\,\oo_{\varphi_m}^n&=&\int_M\;|\nabla\nabla
f|^2\,\oo_{\varphi_m}^n
+\int_M\;Ric(\oo_{\varphi_m})(\Na f_m, \Na f_m)\,\oo_{\varphi_m}^n\\
&\leq &\mu_m+C_1(A, \oo)\int_M\;|\Na f_m|^2\,\oo_{\varphi_m}^n\\
&\leq &\mu_m+C_2(A, \oo)\int_M\; f_m^2\,\oo_{\varphi_m}^n+\frac
12\int_M\;(\Delta_{m}f_m)^2\,\oo_{\varphi_m}^n, \eeqs which implies
that \beq \int_M\;(\Delta_{m}f_m)^2\,\oo_{\varphi_m}^n\leq C_3(A,
\oo). \label{eq0021} \eeq Combining (\ref{eq0022})-(\ref{eq0021}),
there is a subsequence of functions $\{f_m\}$, which we denote by
$\{f_{m_k}\}$, converges weakly in $W^{2, 2}(M,
\oo_{\varphi_{\infty}})$:
$$f_{m_k}\rightharpoonup f_{\infty}, \;\;\hbox{in}\;\; W^{2, 2}(M,
\oo_{\varphi_{\infty}}),\quad k\ri+\infty. $$ By the Sobolev
embedding theorem, we have \beq \Na_{m_k}f_{m_k} \rightarrow
\Na_{\infty} f_{\infty}, \;\;\hbox{in}\;\; L^2(M,
\oo_{\varphi_{\infty}}),\quad \hbox{and}\quad f_{m_k}\rightarrow
f_{\infty} \;\;\hbox{in}\;\; L^2(M, \oo_{\varphi_{\infty}}).
\label{eq0023}\eeq In particular, we have \beq
\int_M\;f_{\infty}^2\,\oo_{\varphi_{\infty}}^n=1,\quad
\int_M\;f_{\infty}\,\oo_{\varphi_{\infty}}^n=0. \label{eq0025}\eeq
By (\ref{eq0023}) the functions $h_{m_k}$ subconverges to a function
$h_{\infty}$ in $W^{3, 2}(M, \oo_{\varphi_{\infty}}),$ which implies
that \beq \int_M\; X(h_{\infty})\,\oo_{\infty}^n=0, \quad
\forall\,X\in \frak h_0(M). \label{eq0024} \eeq On the other hand,
we have \beqs
\int_M\;|\Na_{\infty}\Na_{\infty}f_{\infty}|^2\,\oo_{\varphi_\infty}^n&\leq&
\liminf_{k\ri
+\infty}\;\int_M\;|\Na_{\infty}\Na_{\infty}f_{m_k}|^2\,\oo_{\varphi_\infty}^n\\
&\leq&\liminf_{k\ri
+\infty}\;\int_M\;|\Na_{m_k}\Na_{m_k}f_{m_k}|^2\,\oo_{m_k}^n=0.
\eeqs Thus, $X_{\infty}:=\Na_{\infty}f_{\infty}$ is a holomorphic
vector field in $\frak h_0(M)$, and by (\ref{eq0024}) we have
$$0=\int_M\;X_{\infty}(h_{\infty})\,\oo_{\varphi_\infty}^n=-\int_M\;f_{\infty}\Delta_{\infty}
h_{\infty}\,\oo_{\varphi_\infty}^n=-\int_M\;f_{\infty}^2\,\oo_{\varphi_\infty}^n,$$
which contradicts (\ref{eq0025}). The lemma is proved.\\
\end{proof}

A direct corollary of Lemma \ref{lem:key1} is the following result:

\begin{lem}\label{lem:calabi1}Let $\oo_t(t\in [0, T])$ be a solution
of Calabi flow. Suppose that the Futaki invariant of $[\oo_t]$ vanishes.
If the metric satisfies
$|\varphi(t)|_{C^{4,\al}}\leq A$,  then there is a constant $\mu=\mu(A,
\oo)$ such that
$$\pd {}tCa(t)\leq -\mu Ca(t),\quad \forall t\in [0, T].$$
\end{lem}
\begin{proof}
Direct calculation shows that
$$\frac d{dt}Ca(t)=-\int_M\;|\Na\Na S|^2\,\oo_{\varphi}^n.$$
In Lemma \ref{lem:key1}, choose $f=S-\underline S=\dot\varphi$. Since the Futaki invariant vanishes, the second condition in \eqref{lem:key1con} holds. Therefore, there
is a constant $\mu=\mu(A, \oo)$ such that
$$\int_M\;|\Na\Na S|^2\,\oo_{\varphi}^n\geq \mu \int_M\; (S-\un S)^2\,\oo_{\varphi}^n,
\quad t\in [0, T].$$
The lemma follows directly from the above inequalities.
\end{proof}

We have a similar result for the modified Calabi flow.

\begin{lem}\label{lem:calabi1ex}
Let $\oo_t(t\in [0, T])$ be a solution
of modified Calabi flow. Suppose that the modified
 Futaki invariant of $[\oo_t]$ vanishes.
If the metric satisfies
$|\psi(t)|_{C^{4,\al}}\leq A$, then there is a constant $\mu=\mu(A,
\oo)$ such that
$$\pd {}t\td Ca(t)\leq -\mu \td Ca(t),\quad \forall\, t\in [0, T].$$
\end{lem}
\begin{proof}
By (\ref{epo}) we obtain the evolution of the modified
Calabi energy along the modified Calabi flow,
$$\frac d{dt}\td Ca(t)=-2\int_M\;|\Na\Na \dot\psi|^2\,\oo_{\psi}^n.$$
In Lemma \ref{lem:key1}, let $f=\dot\psi
=S_\psi-\underline S+\theta_{\Re V}(\psi)$, the vanishing modified Futaki invariant guarantee the condition \ref{lem:key1con}. So there
is a constant $\mu=\mu(A, \oo)$ such that
$$\int_M\;|\Na\Na \dot\psi|^2\,\oo_{\psi}^n\geq \mu \int_M\; \dot\psi^2\,\oo_{\psi}^n,
\quad t\in [0, T].$$
The lemma follows directly from \eqref{mcf} and \eqref{mcalabie}.
\end{proof}

\subsection{The first eigenvalue estimate with bounded curvature}
In this subsection, we will show the exponential decay of the Calabi energy on a
given time interval where the Riemann curvature tensor and Sobolev constant are bounded.
Here, we assume that  $M$ has
non-zero holomorphic vector fields. The argument needs the pre-stable condition, which is defined as follows (cf. Chen-Li-Wang \cite{MR2481736},
Phong-Sturm \cite{MR2215459}).
\begin{defi}\label{pre-stable} The complex structure $J$ of $M$
is called pre-stable, if no complex structure of the orbit of
diffeomorphism group contains larger (reduced) holomorphic
automorphism group.

\end{defi}

\begin{lem}\label{lem:key2}Suppose that $(M, J)$ is pre-stable.  For
any $\La, K>0$, if the metric $\oo_\varphi\in \Om$ satisfies \beq
|Rm|(\oo_{\varphi})\leq \La,\quad C_S(\oo_{\varphi})\leq K,
\label{eq100}\eeq then there exists a constant $\mu=\mu(\La, K,
\oo)>0$ such that for any smooth function $f(x)\in C^{\infty}(M)$
with \beq \int_M\;f\,\oo_{\varphi}^n=0,\quad
\int_M\;X(h)\,\oo_{\varphi}^n=0,\quad \forall\,X\in \frak h_0(M),
\eeq where $h$ is a function determined by the equation
$\Delta_{\varphi}h=f,$ the following inequality holds: \beq
\int_M\;|\Na\Na f|^2\,\oo_{\varphi}^n\geq \mu
\int_M\;f^2\,\oo_{\varphi}^n. \eeq

\end{lem}
\begin{proof}
The argument is almost identical to that of  Lemma \ref{lem:key1}. See also the proof of Theorem 4.16 in   Chen-Li-Wang \cite{MR2481736}.
\end{proof}

A direct corollary is the following

\begin{lem}\label{lem:calabi2}Let $\oo_t(t\in [0, T])$ be a solution
of  Calabi flow. Suppose that $(M, J)$ is pre-stable and
 the Futaki invariant of $[\oo_t]$ vanishes.
If the metric satisfies
$|Rm|(t)\leq A$ and $C_S(t)\leq K$ then there is a constant
$\mu=\mu(A, K, \oo)$ such that
$$\pd {}tCa(t)\leq -\mu Ca(t),\quad \forall \,t\in [0, T].$$
\end{lem}

 We have an analogous result for the modified Calabi flow.

\begin{lem}\label{lem:calabi2ex}Let $\oo_t(t\in [0, T])$ be a solution
of modified Calabi flow. Suppose that $(M, J)$ is pre-stable and
 the modified Futaki invariant of $[\oo_t]$ vanishes.
If the metric
satisfies $|Rm|(t)\leq A$ and $C_S(t)\leq K$ then there is a
constant $\mu=\mu(A, K, \oo)$ such that
$$\pd {}t\td Ca(t)\leq -\mu \td Ca(t),\quad \forall \,t\in [0, T].$$
\end{lem}

\subsection{Decay estimates of higher order curvature}
In this subsection, we will show the decay of the higher order derivatives of the scalar curvature under the decay assumption of the Calabi energy. First, we recall the Sobolev inequality:

\begin{lem}\label{lem:sobolev2} For any integer $p>2n$ and K\"ahler metric $\oo$,
there is a constant
 $C_S=C_S( p,
 \oo)$ such that
 \beq
\max_M|f|\leq  C_S\Big(\int_M\;(|f|^p+|\Na f|^p)\,\oo^n\Big)^{\frac
1p}.
 \eeq

\end{lem}

To estimate the higher derivatives of the curvature, we need the
interpolation formula of Hamilton in \cite{MR664497}:
\begin{lem}\label{lem:Ha82} (\cite{MR664497})Let $n=\dim_{\CC}M.$ For any tensor $T$ and $1\leq j\leq
k-1$, we have \beqs \int_M\;|\Na^j T|^{2k/j}\,\omega^n&\leq&C\cdot
\max_M|T|^{\frac {2k}{j}-2}\int_M\;|\Na^k T|^2\,\omega^n,\\
 \int_M\;|\Na^j T|^{2 }\,\omega^n
&\leq &C\Big(\int_M\;|\Na^k T|^2\,\omega^n\Big)^{\frac jk}
 \Big(\int_M\;|  T|^2\,\omega^n\Big)^{1-\frac jk}\eeqs where
$C=C(k, n)$ is a constant.

\end{lem}

Combining Lemma \ref{lem:sobolev2} with Lemma \ref{lem:Ha82}, we have
the following result, which gives the higher order estimates in
terms of the integral norms of the tensor:

\begin{lem}\label{lem:tensor}For any  integer $i\geq 1$ and K\"ahler metric $\oo$,
there exists a constant $C=C(C_S(\oo), i)>0$ such that
 for any tensor $T$, we have
\begin{align*}
\max_M\;|\Na^iT|^2
&\leq C\cdot\max_M\;|T|^{\frac
{2n+1}{n+1}}\Big(\int_M\;|T|^2\,\oo^n\Big)^{\frac
1{4(n+1)}}\\
&\Big(\int_M\;|\Na^{4(n+1)i}T|^2\,\oo^n
+\int_M\;|\Na^{4(n+1)(i+1)}T|^2\,\oo^n\Big)^{\frac
1{4(n+1)}}.
\end{align*}

\end{lem}
\begin{proof}Choosing $f=|\Na^iT|^2$ and $p=2(n+1)$ in Lemma
\ref{lem:sobolev2}, we have \beq \max_M|\Na^iT|^2\leq
C_S\Big(\int_M\;(|\Na^iT|^{4(n+1)}+|\Na
f|^{2(n+1)})\,\oo^n\Big)^{\frac 1{2(n+1)}}.\nonumber\eeq On
From the Kato's inequality we have \beqs |\Na f|=2|\Na^i T|\cdot |\Na|\Na^i
T\|\leq 2|\Na^i T|\cdot |\Na^{i+1} T|
\leq |\Na^i T|^2 +  |\Na^{i+1} T|^2
,\eeqs
also,
\begin{align}\label{eq012}\max_M|\Na^iT|^2
\leq C\Big(\int_M\;(|\Na^iT|^{4(n+1)}+|\Na^{i+1}
T|^{4(n+1)})\,\oo^n\Big)^{\frac 1{2(n+1)}},\end{align}
where $C=C(n, C_S).$ Letting $k=2(n+1)i$, $j=i$ in the first inequality in Lemma
\ref{lem:Ha82} we have
\beqs \int_M\;|\Na^i T|^{4(n+1)}\,\omega^n&\leq&C\cdot
\max_M|T|^{2k+2}\int_M\;|\Na^{2(n+1)i} T|^2\,\omega^n, \\
 \nonumber\eeqs
while $k=2(n+1)(i+1)$, $j=i+1$, we have
\beqs \int_M\;|\Na^{i+1} T|^{4(n+1)}\,\omega^n&\leq&C\cdot
\max_M|T|^{2k+2}\int_M\;|\Na^{2(n+1)(i+1)} T|^2\,\omega^n. \\
 \nonumber\eeqs
Taking $k=4(n+1)i$ and $j=2(n+1)i$ in the second inequality in Lemma
\ref{lem:Ha82} we have  \beqs \int_M\;|\Na^{2(n+1)i}
T|^{2}\,\oo_1^n&\leq& C\cdot \Big(\int_M\;|\Na^{4(n+1)i}T|^2\Big)^{\frac 12} \Big(\int_M\;|T|^2\,\oo^n\Big)^{\frac
12},\\
 \nonumber\eeqs
 Combining this with
(\ref{eq012}), we have \beqs &&\max_M|\Na^iT|^2\\
&\leq&C\cdot\max_M\;|T|^{\frac
{2n+1}{n+1}}\label{eq013}\Big(\int_M\;(|\Na^{2(n+1)i}T|^{2}+|\Na^{2(n+1)(1+i)}
T|^{2})\,\oo^n\Big)^{\frac 1{2(n+1)}}\\
&\leq&C\cdot\max_M\;|T|^{\frac
{2n+1}{n+1}}\Big(\int_M\;|T|^2\,\oo^n\Big)^{\frac
1{4(n+1)}}\Big(\int_M\;|\Na^{4(n+1)i}T|^2\,\oo^n+\int_M\;|\Na^{4(n+1)(i+1)}T|^2\,\oo^n\Big)^{\frac
1{4(n+1)}},\eeqs The lemma is proved.\\
\end{proof}

Now, we can show the decay of higher order derivatives of the scalar curvature.
\begin{lem}\label{lem:decay2}Given constants $\La, K, T>0$ and a positive function
$\ee(t)$. If
$\oo_t(t\in [0, T])$ is a solution
 of the Calabi flow with \beq |Rm|(t)\leq \La, \quad C_S(t)\leq K,
 \quad Ca(t)\leq \ee(t),\quad \forall\,t\in [0, T], \label{eq1010}\eeq then for any integer $i\geq 1$  and any $t_0 \in (0, T)$  there exists a constant $C=C(t_0, i,  \La, K,  n, V)>0$ such that
$$|\Na^iS|\leq C(t_0, i, K, \La,   n, V)\ee(t)^{\frac 1{8(n+1)}},\quad \forall \;t\in [t_0, T].$$

\end{lem}
\begin{proof} By Theorem 3.1 in Chen-He \cite{MR2957626}, for any $t\in [0, T]$ we have the inequality
\begin{align}\label{lem:chenhe3}
\pd {}t\int_M\;|\Na^k Rm|^2\,\oo_t^n\leq -\frac 12\int_M\;|\Na^{k+2}
Rm|^2\,\oo_t^n +C\int_M\;|Rm|^2\,\oo_t^n,
\end{align} where
$C=C( k, \La, n)$ is a constant. To estimate the higher order
derivatives of the curvature tensor, we follow the argument in
Chen-He \cite{MR2957626} to define
$$F_k(t)=\sum_{i=0}^k\,t^i\int_M\;|\Na^i Rm|^2(t)\,\oo_t^n.$$
Using Lemma \ref{lem:Ha82}, we have
$$\pd {}tF_k(t)\leq -\frac 14\sum_{i=0}^k\,t^i\int_M\;|\Na^{i+2}Rm|^2\,\oo_t^n+C\int_M\;
|Rm|^2\,\oo_t^n,$$ where $C=C( k, \La,n).$
This implies that \beq
F_k(t)\leq\int_M\;|Rm|^2(0)\;\oo_0^n+C\int_0^t\,dt\int_M\;|Rm|^2(t)\,\oo_t^n
\leq C(k, \La, n, V)(1+t).\nonumber \eeq   Thus, for any integer $k\geq
1$ we have \beq \int_M\;|\Na^k Rm|^2(t)\,\oo_t^n\leq \frac {C(n, k,
 \La, V)(1+t)}{t^k},\quad \forall \,t\in (0, T]. \nonumber\eeq Combining this
with Lemma \ref{lem:tensor} for the function $T=S-\un S$ and integer $i\geq 1$, we have for any $t_0\in (0, T)$,
\beqs
\max_M\;|\Na^i S|^2&\leq &C(i, K, \La, n, V)Ca(t)^{\frac 1{4(n+1)}}\cdot
\Big(\int_M\; |\Na^{4(n+1)i}S|^2\oo_1^n+\int_M\; |\Na^{4(n+1)(i+1)}S|\Big)^{\frac
1{4(n+1)}}\\
&\leq &C(t_0, i, K, \La, n, V)\ee(t)^{\frac 1{4(n+1)}}, \quad \forall
\,t\in [t_0, T].
\eeqs

\end{proof}

A direct corollary of Lemma \ref{lem:decay2} is the following:

\begin{lem}\label{lem:decay1}Given any constants $A>0, c>1, T>0$ and a positive function
$\ee(t)$.
If
$\oo_t(t\in [0, T])$ is a solution
 of the Calabi flow with \beq \frac 1c\oo\leq \oo_t\leq c\oo, \quad
|\varphi(t)|_{C^{4, \al}}\leq A, \quad Ca(t)\leq \ee(t),\quad \forall
t\in [0, T], \label{eq1010}\eeq then  for any integer $i\geq 1$ and any $t_0\in (0, T)$ there exists $C=C(i, c, A, n, V, \oo)>0$ such that
$$|\Na^iS|\leq C( t_0, i, c, A, n, V, \oo)\ee(t)^{\frac 1{8(n+1)}},\quad \forall \;t\in [t_0, T].$$
\end{lem}

\section{The Calabi flow of the K\"ahler potentials}
\label{Calabi flow of the Kahler potentials}

\subsection{Proof of \thmref{theo:main1}}
\begin{proof}[Proof of \thmref{theo:main1}]
For any $\la,  \La, \ee>0$, if the initial metric $\oo_0:=\oo+\pbp
\varphi_0$ satisfies the condition (\ref{eq0018}), then  by Proposition
\ref{short time} there exist constants $\tau=\tau(\La, \la, \omega)>0$
and constants $c=c(\La, \la, \oo)$, $B=B(\La, \la,  \omega)$ such
that \beqs \oo_{\tau}\in \cA(c, B, \ee):=\Big\{
\oo_{\varphi}\;\Big|\;\frac 1c\oo\leq \oo_\varphi\leq
c\oo,\;\;|\varphi|_{C^{4, \al}}\leq B,\;\; Ca(\oo_{\varphi})\leq \ee
\Big\}. \eeqs Similarly, we can also choose $t_0=t_0(\La, \la, \omega)>\tau$
such that $\oo_{t_0}\in \cA(2c, 2B, \ee).$\\

Now we start from the time $\tau$. Suppose \beq
T:=\sup\Big\{t>0\;\Big|\;\oo_s\in \cA(6c, 6B, \ee),\quad \forall
\;s\in [\tau, t)\Big\}<+\infty.\label{eq:T}\eeq
\begin{lem}\label{lem:extend1}There exists  small $\ee=\ee(c, B)>0$
 such that there is a constant $t_0>\tau$ such that the solution satisfies $$\oo_t\in
\cA(3c, 3B, \ee), \quad  \forall\; t\in [t_0, T].$$
\end{lem}

\begin{proof}
Since the solution $\varphi(t)$ satisfies
$|\varphi(t)|_{C^{4,\alpha}}\leq 6B$ for $t\in [\tau, T)$ and the Futaki invariant
vanishes, by Lemma \ref{lem:calabi1} the Calabi energy decays
exponentially
$$\pd {}tCa(t)\leq -\mu\;Ca(t),\quad t\in [\tau, T]$$
where $\mu=\mu(6B, \oo)$. It follows that $Ca(t)\leq \ee e^{-\mu
(t-\tau)}$ for any $t\in [\tau, T].$ Combining this with Lemma
\ref{lem:decay1}, the higher order derivatives of the scalar curvature have the estimates  \beq
|\Na^iS|(t)\leq C_1(t_0,i,6c, 6B, n,  \oo)\ee^{\frac 1{8(n+1)}}
e^{-\frac {\mu }{ 8(n+1)}(t-t_0) },\quad \forall\;t\in [t_0, T],
\label{eq011}\eeq
where $t_0$ is chosen in step (1).
By the equation of Calabi flow, we have
$$\Big|\pd {}t\oo_t\Big|\leq |\Na\bar\Na S|\leq C_1(t_0, 6c, 6B,  n,  \oo )
\ee^{\frac 1{8(n+1)}} e^{-\frac {\mu }{ 8(n+1)}(t-t_0)}, \quad
\forall\;t\in [t_0, T].$$ It follows that
$$e^{-C_2(t_0,c, B, n,   \oo) \ee^{\frac 1{8(n+1)}} }\oo_{t_0}\leq
\oo_t\leq e^{C_2(t_0,c, B, n,   \oo)  \ee^{\frac
1{8(n+1)}}}\oo_{t_0}, \quad \forall\;t\in [t_0, T].$$ If we choose
$\ee$ sufficiently small, then
$$\frac 1{3c}\oo\leq \oo_t\leq 3c\oo,\quad \forall\;t\in [t_0, T].$$
Now we estimate $|\varphi|_{C^{4, \al}}.$ In fact, along the Calabi
flow we have \beq \pd {}t|\Na^i\varphi|^2\leq c(n)(|\Na\bar\Na
S|\cdot |\Na^i\varphi|^2+|\Na^iS|\cdot |\Na^i\varphi|),
\label{eq110a}\eeq where $c(n)$ is a constant depending only on $n$.
Note that (\ref{eq110a}) can be written as
$$\pd {}t\log(1+|\Na^i\varphi|)\leq c(n)(|\Na\bar\Na S|+|\Na^i S|).$$
 Using the estimate
(\ref{eq011}), we have \beq 1+|\Na^i\varphi|(t)\leq e^{C(t_0,i,6c,
6B, n, V, \oo)  \ee^{\frac 1{8(n+1)}}} (1+|\Na^i\varphi|(t_0)),\quad
\forall \;t\in [t_0, T].\label{eq014}\eeq Since
$\varphi(t_0)\in \cA(2c, 2B, \ee)$,
 the estimates (\ref{eq014}) imply $|\varphi|_{C^{4,
\al}}(t)\leq 3B(t\in [t_0, T])$ for sufficiently small $\ee$. The
lemma is
proved.\\
\end{proof}

By Lemma \ref{lem:extend1}, we can extend the solution $\oo_t$ to
$[0, T+\dd]$  for some $\dd=\dd(c, B)>0$ such that
$$\oo_t\in \cA(6c, 6B,  \ee),\quad \forall\; t\in (t_0, T+\dd],$$
which contradicts (\ref{eq:T}).
 Therefore,  all derivatives of $\varphi$ are uniformly bounded  for all time $t>0$
 by Lemma
\ref{short time} and the Calabi energy decays exponentially by Lemma
\ref{lem:calabi1}. Thus, the Calabi flow converges exponentially
fast to a constant scalar curvature metric. The theorem is proved.

\end{proof}

Similarly, using Lemma \ref{lem:calabi1ex} we have the following
result for the modified Calabi flow. The proof is paralleling to that
of Theorem \ref{theo:main1}, and we omit the details here.

\begin{thm}\label{theo:main1ext}
Let $(M, \oo)$ be a compact K\"ahler manifold  with vanishing
modified Futaki invariant. For any $\la, \La>0$, there is a
 constant $\ee=\ee(\la, \La,\om)$ such
that for any $K$-invariant metric $\oo_{\varphi}\in [\oo]$ satisfying \beq
\oo_{\varphi}\geq \la\,\oo,\quad |\varphi|_{C^{3,\al}}\leq \La,
\quad \td Ca(\oo_{\varphi})\leq \ee,  \eeq   the modified Calabi
flow with the initial metric $\oo_{\varphi}$ exists for all time and
converges exponentially fast to an extremal K\"ahler metric.

\end{thm}

\subsection{Proof of \thmref{theo:main1a} and Corollary \ref{cor:main}}

\begin{proof}[Proof of Theorem \ref{theo:main1a}]
 By the assumptions of Theorem \ref{theo:main1a},
the metric $\oo_{0}=\oo+\pbp\varphi_0$ satisfies the following
conditions \beq  Ric(\oo_0)\geq -\la \oo_0,\quad
\|S\|_{L^p(\oo_0)}\leq \La, \quad C_S(\om_0)\leq K,  \quad
\nu_{\oo}(\oo_0)\leq \inf_{\oo'\in \Om}\nu_{\oo}(\oo')+\ee.
\label{eq401}\eeq Therefore, according to \thmref{compactness} the
corresponding K\"ahler potential \beq |\vphi_0|_{C^{3,\al}}\leq
C(\la, \La, K,\oo),\quad \oo_{\varphi_0}\geq c\, \oo>0.
\label{eq402}\eeq We use the Calabi flow to smooth this
K\"ahler potential. Let $\vphi(t)$ be the Calabi flow with
the initial data $\varphi_0$. According to Proposition \ref{short time},
there exists $t_0>0$ such that $|\vphi(t_0)|_{C^{4,\al}}$ is bounded by a constant
depending on $\om$ and $|\vphi_0|_{C^{3,\alpha}}$. Moreover,
$\varphi(t_0)$ satisfies
$$\oo_{\varphi(t_0)}\geq \tilde c\cdot \oo$$ for some constant $\tilde c>0$. Since
the $K$-energy is decreasing along the Calabi flow,
 the $K$-energy of $\vphi(t_0)$ is not larger than $ \inf_{\oo'\in
\Om}\nu_{\oo}(\oo')+\ee$.

  In order to apply Theorem \ref{theo:main1}, we need
to show the Calabi energy of $\varphi(t_0)$ is small. This follows
from the next lemma.
\begin{lem}
Assume that $\vphi$ satisfies $|\vphi|_{C^{4,\al}}\leq B$ for a
fixed constant $B$. Then for any constant $\delta>0$, there is
positive constant $\eps$ depending on $B$ and $\delta$ such that if
$\nu_{\oo}(\oo_0)\leq \inf_{\oo'\in \Om}\nu_{\oo}+\ee$, then
$Ca(\vphi)\leq\delta$.
\end{lem}
\begin{proof}
We argue this by contradiction. Suppose not, there are $\dd_0>0$ and
a sequence of positive constants $\ee_i\ri 0$ and $\varphi_i$
satisfying
$$Ca(\varphi_i)> \dd_0>0, \quad |\vphi_i|_{C^{4,\al}}\leq B,  \quad \oo_{\varphi_i}
\geq c\,\oo, \quad \nu_{\oo}(\oo_i)\leq \inf_{\oo'\in
\Om}\nu_{\oo'}+\ee_i.$$ Thus, we can take a subsequence of
$\varphi_i$ such that $\varphi_i\ri \varphi_{\infty}$ in $C^{4,
\al'}$ for some $\al'<\al$ with
$\nu_{\oo}(\oo_{\infty})=\inf_{\oo'\in \Om}\nu_{\oo'}$. Thus,
$\oo_{\infty}$ is a $C^{4,\al'}$ cscK metric with
$Ca(\varphi_{\infty})\geq \dd_0>0$, which is impossible.
\end{proof}
Therefore, we apply \thmref{theo:main1} to obtain Theorem
\ref{theo:main1a}.

\end{proof}

Using the "modified" version of Theorem \ref{compactness},
the proof of Theorem \ref{theo:main1a} also works for the
modified Calabi flow and extK metrics. Thus, we have the result:

\begin{theo}\label{theo:main1aext}
Let $(M, \oo)$ be a compact K\"ahler manifold for which the modified
$K$-energy is $I$-proper.
 For any constants $\la,  \La, Q>0$ and $p>n$, there is a constant $\ee=\ee(\la,
 \La, Q, p,
 \om)$ such
that if there exists a $K$-invariant metric $\oo_{0}\in [\oo]$ satisfying the
following conditions
\begin{align*}
& Ric(\oo_0)\geq -\la \,\oo_0,\quad
\|S(\om_0)\|_{L^p(\oo_0)}\leq \La, \quad C_S(\om_0)\leq Q,  \quad\tilde\nu_{\oo}(\oo_0)\leq \inf_{\oo'\in
\Om}\tilde\nu_{\oo}(\oo')+\ee,
\end{align*}
then the modified Calabi flow with the  initial metric
$\oo_0$ exists for all time and converges exponentially fast to an
extK metric.\\
 \end{theo}

Combining Theorem \ref{compactness} with Theorem \ref{theo:main1},
we have

\begin{proof}[Proof of Corollary \ref{cor:main}] Since the
$K$-energy is decreasing along the Calabi flow, all the conditions
in Theorem \ref{compactness} are satisfied for the evolving metrics.
By Proposition  \ref{short time} the Calabi flow has long time solution and
 there exist uniform constants $\la, C>0$ such that for any $t\in (t_0, \infty)(t_0>0),$
\beq \oo_t\geq \la\,\oo, \quad |\varphi(t)|_{C^{k, \al}}\leq C. \label{eqb1}\eeq

We claim that $\lim_{t\ri +\infty}Ca(t)=0$. In fact, by (\ref{eqb1}) there is a
sequence of metrics $\oo_{t_i}(t_i\ri+\infty)$ converging to a limit metric
$\oo_{\infty}$ smoothly and $\oo_{\infty}$ satisfies
$$\int_M\;|\Na\Na S(\oo_{\infty})|^2\;\oo_{\infty}^n=0.$$
Therefore, $\oo_{\infty}$ is an extK metric. Since the $K$-energy is
proper, the Futaki invariant vanishes and
$\oo_{\infty}$ is a cscK metric. Since the Calabi energy is
decreasing along the flow, we have $\lim_{t\ri+\infty}Ca(t)=0.$

Thus, the Calabi energy is sufficiently small when $t$ is large
enough. By Theorem \ref{theo:main1}, the flow converges exponentially fast to a cscK
metric. Similar arguments also work for the modified Calabi flow,
and we omit the details here.\\
\end{proof}

The proof of Corollary \ref{cor:main} also works for the modified
Calabi flow and extremal K\"ahler metrics. Here we state the result
and omit its proof.
\begin{cor}\label{cor:mainex}
Let $(M, \Om)$ be an $n$-dimensional compact K\"ahler manifold for
which the modified $K$-energy is $I$-proper in the K\"ahler class
$\Om.$  If the $L^p$-norm of the  scalar curvature for
some $p>n$, the Sobolev constant and the lower bound of Ricci
curvature are uniformly bounded along the flow, then the modified
Calabi flow with a K-invariant initial metric converges exponentially fast to an extK
metric.\\

\end{cor}

On Fano surfaces the conditions on Sobolev constant in Corollary
\ref{cor:main} can be removed.  It was
 observed in Tian-Viaclovsky \cite{MR2166311} that the Sobolev
  constant of the cscK metric is essentially bounded by the positive Yamabe
  invariant. The Yamabe invariant is positive when the K\"ahler
  class $\Om$ stays in the interior of Tian's cone
\begin{align*}
3C_1^2>2\frac{(C_1 \cdot \Om)^2}{\Om^2}.
\end{align*}
Here $C_1$ denotes the first Chern class $C_1(M)$ for simplicity.
This idea was generalized to the extK metrics (cf. Chen-Weber \cite{MR2737786} and
 Chen-Lebrun-Weber \cite{MR2425183}) in generalized Tian's cone
\begin{align*}
48\pi^2 C_1^2>A(\Om):=32\pi^2(C_1^2+\frac{1}{3}\frac{(C_1\cdot \Om)^2}{\Om^2})+\frac{1}{3} \| \mathcal F \|^2.
\end{align*}
The last term is the norm of the Calabi-Futaki invariant \cite{MR2471594}\cite{MR1314584}.
Chen-He \cite{MR2644371} proved in Lemma 2.3 that if the minimum of the extremal Hamiltonian potential of the all $K$-invariant K\"ahler metric is positive and the Calabi energy of the initial $K$-invariant K\"ahler metric is less than $A(\Om)$,
then the Sobolev constant is bounded.
That shows that on Fano surfaces, the third condition in our \thmref{compactness} could be verified when $Ca<A(\Om)$.
Combining   the Sobolev constants with the classification of the blowing up bubbles,  Chen-He \cite{MR2644371} \cite{MR2957626} proved the convergence
  of the Calabi flow on the toric Fano surfaces under the restriction on the initial Calabi energy.
Our next
corollary shows that the curvature
 obstructions of the convergence are  the $L^p$-norm of the scalar curvature and
 the lower bound of the Ricci curvature.

\begin{cor}\label{cor:main2}
On a compact Fano surface for which the modified $K$-energy is
$I$-proper. Assume that the initial K\"ahler metric is $K$-invariant
and has Calabi energy less than $A(\Om)$. If the $L^p$-norm of the
scalar curvature for some $p>2$ and the lower bound of the
Ricci curvature are uniformly bounded along the flow, then the
Calabi flow converges exponentially fast to an extK metric.
\end{cor}

\section{Calabi flow of the K\"ahler metrics}\label{Calabi flow of the Kahler metrics}

\subsection{Proof of Theorem \ref{theo:main2}}
\begin{proof}[Proof of Theorem \ref{theo:main2}]
  Given any constants $\La,  K,
\dd, \ee>0$, we define
$$\cB(\La, K, \dd, \ee)=\Big\{\oo_{\varphi}\;\Big|\;|Rm|(\oo_{\varphi})
\leq \La, \; C_S(\oo_{\varphi})\leq K, \; \mu_1(\oo_{\varphi})\geq
\dd, \;Ca(\oo_{\varphi})\leq \ee\Big\}.$$ By the assumption, we
assume that the Calabi flow with the initial metric $\oo_0$ is of
type $II(\tau, \La)$ and $\oo_0\in \cB(\La, K, \dd, \ee).$

\begin{lem}\label{lem:a1}There is a $\tau_0=\tau_0(\tau, \La, K, \dd)>0$ such that for any $t\in [0, \tau_0]$  we have
$$\oo_t\in \cB(2\La, 2K,  \frac {2\dd}3,  \ee).$$
\end{lem}
\begin{proof} Since $\oo_t$ is of type  $(\tau, \La)$, we have
$|Rm|(t)+|\Na\bar \Na S|\leq \La$ for any $t\in [0, \tau]$.  Thus,
by the equation of Calabi flow the evolving  metrics satisfy the
inequality
$$\Big|\pd {g_{i\bar j}}t\Big|\leq \La,\quad \forall\;t\in [0, \tau].$$
Therefore, we have
$$e^{-\La t}\oo_0\leq \oo_t\leq e^{\La t}\oo_0, \quad \forall\;t\in [0, \tau].$$
Clearly, we can choose $\tau_0\in (0, \tau)$ small such that $C_S(\oo_t)\leq
2K$ for any $t\in [0, \tau_0].$ By Proposition \ref{lem:eigen2} the first
eigenvalue $\mu_1(t)\geq \frac 23\dd$ for $t\in [0, \tau_0]$ when
$\tau_0$ is small. The lemma is proved.
\end{proof}

 To extend the solution, we define
\beq T:=\sup\Big\{t>0\;|\;\oo_t\in \cB(6\La, 6K,   \frac 14\dd,
  \ee)\Big\}. \label{eq:T2}\eeq
Suppose $T<+\infty.$ We have the following lemma.

\begin{lem}\label{lem:a2}There exists $\ee_0=\ee_0(\La, K,  \dd, n)>0$
 such that
 $$\oo_t\in \cB(3\La, 3K,   \frac 12\dd,
  \ee_0),\quad \forall\;t\in [0, T].$$

\end{lem}
\begin{proof}
 Since  the eigenvalue $\mu_1(t)\geq \frac 14\dd$ for any $t\in [0, T]$,
the Calabi energy decays exponentially:
$$Ca(t)\leq \ee e^{-\frac 14\dd t},\quad \forall\,t\in[0, T].$$
By Lemma \ref{lem:decay2}, for any integer $i\geq 1$ we have
$$|\Na^iS|(t)\leq C(\tau_0, i, 6\La, 6K, n )\ee^{\frac 1{8(n+1)}}
e^{-\frac {\dd t}{ 32(n+1)} },\quad \forall\;t\in [\tau_0, T].$$
Thus, using  the equation of Calabi flow as in Lemma \ref{lem:a1} we
can choose $\ee$ small such that $C_S(t)\leq 3K$ for all $t\in [0,
T].$ On the other hand, using Proposition \ref{lem:eigen2}    we have
$$\mu_1(t)+(6\La)^2\geq \Big(\mu_1(\tau_0)+
(6\La)^2\Big)e^{-C(\tau_0,  6\La, 6K, n )\ee^{\frac 1{8(n+1)}}
\int_{\tau_0}^t\,e^{-\frac {\dd t}{ 32(n+1)} dt}}.$$ This implies
that when $\ee$ is small enough, we have the inequality
$$\mu_1(t)\geq \frac 12\dd,\quad \forall\;t\in[0, T].$$
By the evolution equation \eqref{evorm} of $Rm$  we have
$$\pd {}{t}|Rm|^2(t)=Rm*\Na^4S+Rm*Rm*\Na^2S.$$
Therefore, the following inequality holds:
$$\pd {}{t}|Rm|(t)\leq |\Na^4S|+|Rm\|\Na^2S|.$$
It follows that
$$|Rm|(t)\leq |Rm|(\tau_0)+C(\tau_0, \La, K, n,  \dd)\cdot \ee^{\frac 1{8(n+1)}}\leq 3\La,\quad \forall \,t\in [\tau_0, T],$$
where we chosen $\ee$ small  in the last inequality.
 Thus, we have
$\oo_t\in \cB(3\La, 3K,   \frac 12\dd,
  \ee_0)$ for all $t\in [\tau_0, T].$ The lemma is proved.
\end{proof}

  By Lemma \ref{lem:a2} we can extend the solution
$\oo_t$ to $[0, T+\dd']$  for some $\dd'>0$ such that
$$\oo_t\in \cB(6\La, 6K,   \frac 14\dd,
  \ee_0),\quad \forall\; t\in [0, T+\dd'],$$
which contradicts the definition (\ref{eq:T2}).
 Therefore,  all derivatives of $\varphi$ are uniformly bounded  for all time $t>0$
 by Lemma
\ref{short time} and the Calabi energy decays exponentially by Lemma
\ref{lem:calabi1}. Thus, the Calabi flow converges exponentially
fast to a constant scalar curvature metric. The theorem is proved.

\end{proof}

\subsection{Proof of Theorem \ref{theo:main3}}

\begin{proof}[Proof of Theorem \ref{theo:main3}] The proof is almost the same as
that of Theorem \ref{theo:main2}. Here we sketch the details.

 Given constants $\La,  K,
\dd>0$, we define the set
$$\cC(\La, K, \ee)=\Big\{\oo_{\varphi}\;\Big|\;|Rm|(\oo_{\varphi})
\leq \La, \; C_S(\oo_{\varphi})\leq K, \;
 \;Ca(\oo_{\varphi})\leq \ee\Big\}.$$
By the assumption, we assume that the Calabi flow with the initial
metric $\oo_0$ is of type $I(\tau, \La)$ and $\oo_0\in \cC(\La, K,
\ee)$. Following the same argument as in the proof of Lemma
\ref{lem:a1}, we have

\begin{lem}\label{lem:b1}There is a $\tau_0=\tau_0(\tau, \La, K)>0$ such that for any $t\in [0, \tau_0]$  we have $\oo_t\in
\cC(2\La, 2K,  \ee).$
\end{lem}

  To extend the solution, we define
\beq T:=\sup\Big\{t>0\;|\;\oo_t\in \cC(6\La, 6K,
  \ee_0)\Big\}. \label{eq:T3}\eeq
Suppose $T<+\infty.$ We have the following lemma.

\begin{lem}\label{lem:b2}There exists $\ee_0=\ee_0(\La, K,    \oo)>0$
 such that for any $t\in [0, T]$ we have
 $\oo_t\in \cC(3\La, 3K,
  \ee_0).$

\end{lem}
\begin{proof}
 Since $(M, J)$ is pre-stable, by Lemma \ref{lem:calabi2} the Calabi
energy decays exponentially:
$$Ca(t)\leq \ee e^{-\dd t},\quad \forall\,t\in[0, T],$$
where $\dd=\dd(6\La, 6K, \oo)>0.$ By Lemma \ref{lem:decay2}, for any
integer $i\geq 1$ we have \beq |\Na^iS|(t)\leq C(\tau_0, i, 6\La,
6K, n,  \oo )\ee^{\frac 1{8(n+1)}} e^{-\frac {\dd
 t}{
8(n+1)} },\quad \forall\;t\in [\tau_0, T]. \label{eq404}\eeq Thus,
using  the equation of Calabi flow as in Lemma \ref{lem:a1} we can
choose $\ee$ small such that $C_S(t)\leq 3K$ for all $t\in [0, T].$
On the other hand, using the inequality (\ref{eq404}) as in Lemma
\ref{lem:a2}, we show  that
$$|Rm|(t)\leq |Rm|(\tau_0)+C(\tau_0, \La, K, n, \oo)\cdot \ee^{\frac 1{8(n+1)}}\leq 3\La,\quad \forall \,t\in [\tau_0, T],$$
where we chosen $\ee$ small enough in the last inequality.
 Thus, we have
$\oo_t\in \cC(3\La, 3K,
  \ee_0)$ for all $t\in [\tau_0, T].$ The lemma is proved.
  \end{proof}

 By Lemma \ref{lem:b2} we can extend the solution
$\oo_t$ to $[0, T+\dd']$  for some $\dd'>0$ such that
$$\oo_t\in \cC(6\La, 6K,
  \ee_0),\quad \forall\; t\in [0, T+\dd'],$$
which contradicts the definition (\ref{eq:T3}).
 Therefore,  all derivatives of $\varphi$ are uniformly bounded  for all time $t>0$
 by Lemma
\ref{short time} and the Calabi energy decays exponentially by Lemma
\ref{lem:calabi1}. Thus, the Calabi flow converges exponentially
fast to a constant scalar curvature metric. The theorem is proved.

\end{proof}

We have the following analogous result for the modified Calabi flow.
\begin{theo}\label{theo:main3ex}Let $(M, \oo)$ be an $n$-dimensional compact K\"ahler
manifold.  Assume that $M$ is pre-stable.
 For any $\tau,
\La, K>0$, there is a constant $\ee=\ee(\tau, \La, K, n, \oo)>0$
such that if the solution $\oo_t$ of the Calabi flow with any
$K$-invariant
initial metric
 $\oo_0\in [\oo]$ satisfies the following properties:
 \begin{enumerate}
   \item[(a)] the modified Calabi flow $\oo_t$ is of type $(\tau, \La)$;
   \item[(b)] $
 C_S(\oo_0)\leq K, \quad \td Ca(\oo_0)\leq \ee,
 $
 \end{enumerate}
 the Calabi flow $\oo_t$  exists for all the time and converges exponentially fast
to an extremal K\"ahler metric.\\

\end{theo}

\subsection{Discussion of $(\tau, \La)$ Calabi flow}
In this subsection, we give some conditions on the initial metric
such that the Calabi flow is of type $I(\tau, \La)$ or $II(\tau,
\La).$ We introduce some definitions.

\begin{defi}\label{defi:extend}For any $\La_1, \La_2, K>0$, we define  $\cF(\La_1, \La_2, K)$ the set of K\"ahler metrics $\oo_0\in
\Om$ satisfying the following conditions
\begin{enumerate}
  \item[(1)] $\sum_{i=0}^2|\Na^iRm|(\oo_0)\leq \La_1;$
  \item[(2)] $\sum_{i=1}^N\int_M\;|\Na^iRm|^2(\oo_0)\leq \La_2$ where $N=28(n+1);$
  \item[(3)] $C_S(\oo)\leq K.$
\end{enumerate}
\end{defi}

\begin{lem}\label{lem:initial}Given $\La_1, \La_2, K>0$. For any $\oo_0\in \cF(\La_1, \La_2,
K)$
there exists $\tau=\tau(\La_1, \La_2, K, \oo)$ and $\La(\La_1, \La_2, K, \oo)$
 such that the
solution $\oo_t$ of Calabi flow with the initial metric $\oo_0$ is of type $(\tau, \La).$ In other
words, $\oo_t$ satisfies
$$|Rm|(t)+|\Na \bar \Na S|(t)\leq \La,\quad t\in [0, \tau].$$
Moreover, we can choose $\tau$ small such that
$$|Rm|(t)\leq 2|Rm|(0),\quad C_S(t)\leq 2C_S(0),\quad t\in [0, \tau]. $$
\end{lem}

\begin{proof} Let
$$T=\sup\Big\{t>0\;\Big|\; \oo_s\in \cF(2\La_1, 2\La_2, 2K)\quad  \forall\; s\in [0, t]\Big\}.$$
We would like to give a lower bound of $T.$
\begin{enumerate}
  \item[(1)] To estimate the
metric $\oo_t$, we observe that
$$\Big|\pd {}t\oo_t\Big|=|\Na\bar\Na S|\leq 2\La_1,\quad  \forall\; t\in [0, T]$$
which implies that
$$e^{-2\La_1 t}\oo_0\leq \oo_t\leq e^{2\La_1 t}\oo_0, \quad  \forall\; t\in [0, T].$$
Thus, if $t\leq C(\La_1, K)$, we have $C_S(t)\leq 2K.$

\item[(2)]Since the curvature tensor
is  bounded for $t\in [0, T],$ by  (\ref{lem:chenhe3})  we have
\beqs &&\sum_{k=0}^N\int_M\;|\Na^kRm|^2(t)\,\oo_t^n\\&\leq&
\sum_{k=0}^N\int_M\;|\Na^kRm|^2(0)\,\oo_0^n+(N+1)C(n, k,
2\La_1)V(2\La_1)^2t \\&\leq& \La_2+(N+1)C(n, k,
2\La_1)V(2\La_1)^2t,\quad \forall\; t\in [0, T].\eeqs Thus,
$\sum_{i=1}^N\,\int_M\;|\Na^kRm|^2(t)\,\oo_t^n\leq 2\La_2$ if
$$t\leq \frac {\La_2}{ (N+1)C(n, k,
2\La_1)V(2\La_1)^2}.$$
\item[(3)]By Lemma \ref{lem:tensor} there is a constant $ C_1(c, i, \oo, \La_1, \La_2)$ such
that \beq \max_M|\Na^i S|(t)\leq C_1(c, i, \oo, \La_1, \La_2), \quad
1\leq i\leq 6,\quad t\in [0, T]. \label{eq405}\eeq By the evolution
equation \eqref{evorm} of $Rm$ ,  we have the inequality
$$\pd {}{t}|Rm|(t)\leq |\Na^4S|+|Rm\|\Na^2S|.$$
It follows that
$$|Rm|(t)\leq |Rm|(0)+C_2(c,  \oo, \La_1, \La_2)t,\quad \forall \,t\in [0, T].$$
We have similar estimates for higher order derivatives of $Rm$:
$$\pd {}t|\Na^i Rm|\leq |\Na^2R\|\Na^i Rm|+|\Na^{i+4}S|,\quad i=1, 2.$$
Therefore, we have
$$\sum_{i=0}^2|\Na^i Rm|(t)\leq \sum_{i=0}^2|\Na^i Rm|(0)+C_3(c,  \oo, \La_1,
\La_2)t\leq 2\La_1, \quad t\in \Big[0, \frac {\La_1}{C_3}\Big]. $$\\

\end{enumerate}
  Combining the above estimates, we have
$$T\geq \min\Big\{C(\La_1, K),\;\frac {\La_2}{ (N+1)C(n, k,
2\La_1)V(2\La_1)^2}, \;\frac {\La_1}{C_3}\Big\}.$$
The lemma is proved.\\
\end{proof}

Therefore, we can replace the condition $(a)$ in Theorem
\ref{theo:main2} and Theorem \ref{theo:main3} by assuming the
initial metric in $\cF(\La_1, \La_2, K).$ It is interesting to find
a simpler condition to replace the type $(\tau, \La)$ condition.



\end{document}